\documentclass[a4paper, reqno, 10pt]{amsart}

\usepackage{verbatim}
\usepackage[dvips]{color}
\usepackage[usenames,dvipsnames]{xcolor}
\usepackage{amssymb}
\usepackage[active]{srcltx}
\usepackage[colorlinks,pdfpagelabels,pdfstartview = FitH,bookmarksopen = true,bookmarksnumbered = true,linkcolor = NavyBlue,plainpages = false,hypertexnames = false,citecolor = OliveGreen] {hyperref}
\usepackage[italian, textsize=tiny, color=BlueGreen, backgroundcolor=BlueGreen, linecolor=BlueGreen]{todonotes}
\usepackage{mathtools}
\usepackage{empheq}
\usepackage{a4wide}
\usepackage{aligned-overset}
\usepackage{kotex}

\newcommand\blfootnote[1]{%
\begingroup
\renewcommand\thefootnote{}\footnote{#1}%
\addtocounter{footnote}{-1}%
\endgroup
}

\makeatletter
\@namedef{subjclassname@2020}{\textup{2020} Mathematics Subject Classification}
\makeatother

\allowdisplaybreaks

\title[Double phase problems with measure data]{Gradient estimates for degenerate elliptic measure data problems with double phase}

\author[Song]{Kyeong Song}
\address{School of Mathematics,
Korea Institute for Advanced Study, Seoul 02455, Republic of Korea}
\email{kyeongsong@kias.re.kr}

\author[Youn]{Yeonghun Youn}
\address{Department of Mathematics Education,
Incheon National University, Incheon 21999, Republic of Korea}
\email{yeonghunyoun@inu.ac.kr}

\subjclass[2020]{35B65;  
35J70; 
35R05; 
35R06. 
}

\keywords{Double phase problem; Measure data; SOLA; Gradient estimate}

\newtheorem{theorem}{Theorem}[section]

\newtheorem{proposition}[theorem]{Proposition}
\newtheorem{lemma}[theorem]{Lemma}

\theoremstyle{definition}
\newtheorem{definition}[theorem]{Definition}
\newtheorem{remark}[theorem]{Remark}

\numberwithin{equation}{section}

\def\eqn#1$$#2$${\begin{equation}\label#1#2\end{equation}}
\def\charfn_#1{{\raise1.2pt\hbox{$\chi_{\kern-1pt\lower3pt\hbox{{$\scriptstyle#1$}}}$}}}

\makeatletter
\newcommand{\pushright}[1]{\ifmeasuring@#1\else\omit\hfill$\displaystyle#1$\fi\ignorespaces}
\newcommand{\pushleft}[1]{\ifmeasuring@#1\else\omit$\displaystyle#1$\hfill\fi\ignorespaces}
\makeatother

\DeclareMathOperator*{\osc}{osc}

\DeclareMathOperator*{\data}{\mathtt{data}}

\def\loc{{\operatorname{loc}}}

\newcommand{\dx}{\,dx}

\def\mean#1{\mathchoice%
          {\mathop{\kern 0.2em\vrule width 0.6em height 0.69678ex depth -0.58065ex
                  \kern -0.8em \intop}\nolimits_{\kern -0.4em#1}}%
          {\mathop{\kern 0.1em\vrule width 0.5em height 0.69678ex depth -0.60387ex
                  \kern -0.6em \intop}\nolimits_{#1}}%
          {\mathop{\kern 0.1em\vrule width 0.5em height 0.69678ex
              depth -0.60387ex
                  \kern -0.6em \intop}\nolimits_{#1}}%
          {\mathop{\kern 0.1em\vrule width 0.5em height 0.69678ex depth -0.60387ex
                  \kern -0.6em \intop}\nolimits_{#1}}}

\newcommand{\vertiii}[1]{{\left\vert\kern-0.25ex\left\vert\kern-0.25ex\left\vert #1 
			\right\vert\kern-0.25ex\right\vert\kern-0.25ex\right\vert}}
			
\def\avenorm#1{\mathchoice%
          {\mathop{\kern 0.2em\vrule width 0.6em height 0.69678ex depth -0.58065ex
                  \kern -0.545em \|{#1}\|}}%
          {\mathop{\kern 0.1em\vrule width 0.5em height 0.69678ex depth -0.60387ex
                  \kern -0.495em \|{#1}\|}}%
          {\mathop{\kern 0.1em\vrule width 0.5em height 0.69678ex depth -0.60387ex
                  \kern -0.495em \|{#1}\|}}%
          {\mathop{\kern 0.1em\vrule width 0.5em height 0.69678ex depth -0.60387ex
                  \kern -0.495em \|{#1}\|}}}

\newtoks\by
\newtoks\paper
\newtoks\book
\newtoks\jour
\newtoks\yr
\newtoks\pages
\newtoks\vol
\newtoks\publ

\def\ota{{\hbox{\bf ???}}}
\def\cLear{\by=\ota\paper=\ota\book=\ota\jour=\ota\yr=\ota
\pages=\ota\vol=\ota\publ=\ota}
\def\endpaper{\the\by, \textit{\the\paper},
{\the\jour} \textbf{\the\vol} (\the\yr), \the\pages.\cLear}
\def\endbook{\the\by, \textit{\the\book},
\the\publ, \the\yr.\cLear}
\def\endpap{\the\by, \textit{\the\paper}, \the\jour.\cLear}
\def\endproc{\the\by, \textit{\the\paper}, \the\book, \the\publ,
\the\yr, \the\pages.\cLear}

\begin{document}

\begin{abstract}
We study nonlinear elliptic equations modeled on
\[ -\mathrm{div}\,(|Du|^{p-2}Du+a(x)|Du|^{q-2}Du) = \mu, \]
where $2\le p<q<\infty$, $a(\cdot) \ge 0$, and $\mu$ is a signed Borel measure with finite total mass.
We prove local Calder\'on--Zygmund type gradient estimates for SOLA (Solutions Obtained as Limits of Approximations) by finding new and natural assumptions on $p$, $q$ and $a(\cdot)$. 
\end{abstract}

\blfootnote{K. Song was supported by a KIAS individual grant (MG091702) at Korea Institute for Advanced Study. Y. Youn was supported by Incheon National University Research Grant in 2025 (No. 2025-0175).}

\maketitle


\section{Introduction}
In this paper, we investigate the following Dirichlet problem:
\begin{equation}\label{main.eq}
\left\{
\begin{aligned}
-\mathrm{div}\,A(x,Du) &= \mu &\text{in }& \Omega, \\
u&= 0& \text{on }& \partial\Omega.
\end{aligned}
\right.
\end{equation}
Here, $\Omega \subset \mathbb{R}^{n}$ ($n \ge 2$) is a bounded domain and $\mu$ is a signed Borel measure on $\Omega$ with finite total mass $|\mu|(\Omega)<\infty$; in the following, we extend $\mu$ to $\mathbb{R}^{n}$ by letting $|\mu|(\mathbb{R}^{n}\setminus\Omega) = 0$.
The vector field $A: \Omega\times\mathbb{R}^{n} \to \mathbb{R}^{n}$ is assumed to be continuously differentiable with respect to the second variable $z \in \mathbb{R}^{n}$, with $\partial A(\cdot) \equiv \partial_{z}A(\cdot)$ being Carath\'eodory regular. Moreover, it satisfies the following growth, ellipticity and continuity assumptions:
\begin{equation}\label{growth}
\left\{
\begin{aligned}
|A(x,z)| + |\partial A(x,z)||z| &\le L(|z|^{p-1} + a(x)|z|^{q-1}),\\
\nu (|z|^{p-2} + a(x)|z|^{q-2})|\zeta|^{2} &\le \partial A(x,z) \zeta \cdot \zeta , \\
|A(x_{1},z)-A(x_{2},z)| & \le L|a(x_{1})-a(x_{2})||z|^{q-1} ,
\end{aligned}
\right.
\end{equation}
for any $x,x_{1},x_{2} \in \Omega$ and $z , \zeta \in \mathbb{R}^{n}$, where $0<\nu  \le L <\infty$ are fixed constants and $a:\Omega \rightarrow [0,\infty)$ is a modulating coefficient satisfying
\begin{equation*}
0 \le a(x) \le \|a\|_{L^{\infty}} \qquad \text{for a.e.}\;\; x\in\Omega. 
\end{equation*}
Throughout this paper, unless otherwise specified, we assume
\begin{equation}\label{p.bound} 
2 \le p \le n \qquad \text{and} \qquad p < q < \infty.
\end{equation}

Assumptions \eqref{growth} are modeled on the example
\begin{equation}\label{model}
 -\mathrm{div}(p|Du|^{p-2}Du + a(x)q|Du|^{q-2}Du) = \mu \quad \text{in}\;\; \Omega.
\end{equation}
In this case, the operator in the left-hand side of \eqref{model} naturally appears in the minimization problem of the double phase functional
\begin{equation}\label{dp.ftnal}
w \mapsto \int_{\Omega}\left(|Dw|^{p} + a(x)|Dw|^{q}\right)\,dx.
\end{equation}
Such a functional was first studied by Zhikov \cite{Zhi86,Zhi95} in the contexts of Lavrentiev phenomena and homogenizations of strongly anisotropic materials. 
A main feature of the double phase problem is the drastic change of its growth and ellipticity according to the position of $x\in\Omega$. 
More precisely, it is a kind of nonuniformly elliptic problem in the following sense: by letting
\[ A_{\mathcal{P}}(x,z) \coloneqq p|z|^{p-2}z + a(x)q|z|^{q-2}z \]
for $x \in \Omega$ and $z \in \mathbb{R}^n$, 
if $B_{R} \cap \{a(x)=0\} \neq \emptyset$ for a ball $B_R \equiv B_R(x_0)$ and $a(x) \approx |x-x_0|^{\alpha}$, then the (nonlocal) ellipticity ratio
\begin{equation}\label{ellipticity}
\frac{\sup_{x\in B_R}\text{highest eigenvalue of }\partial A_{\mathcal{P}}(x,z)}{\inf_{x\in B_R}\text{lowest eigenvalue of }\partial A_{\mathcal{P}}(x,z)} \approx 1 + R^{\alpha}|z|^{q-p},
\end{equation}
is unbounded with respect to the gradient variable $z$. 
In this point of view, the double phase problem is one of the prominent examples of non-autonomous problems featuring nonstandard growth and nonuniform ellipticity, see \cite{MR21} for a comprehensive overview.

In particular, since the seminal papers by Colombo and Mingione \cite{CM15a,CM15b,CM16JFA}, sharp regularity results for weak solutions to \eqref{main.eq} or minimizers of \eqref{dp.ftnal} have been extensively studied, see for instance \cite{BB,BCM15,BCM18,BL22,BO,DM19,Ok17CV} and references therein. A central assumption for the validity of such regularity results is the following: $a \in C^{0,\alpha}(\Omega)$ for some $\alpha \in (0,1]$ and 
\begin{equation}\label{rate.weak.sol}
\frac{q}{p} \le 1+ \frac{\alpha}{n},
\end{equation}
whose sharpness is shown in \cite{ELM}. 
We also mention the recent paper \cite{BBK}, in which self-improving properties of very weak solutions were proved for double phase systems with data in divergence form under an assumption similar to \eqref{rate.weak.sol}.

In view of \eqref{model} and \eqref{dp.ftnal}, we denote
\begin{equation}\label{def.H}
H(x,t) \coloneqq t^{p} + a(x)t^{q} \qquad \text{and} \qquad h(x,t) \coloneqq t^{p-1} + a(x)t^{q-1}
\end{equation}
for $x \in \Omega$ and $t \ge 0$. Basic properties of the functions $H$ and $h$, along with related function spaces, will be discussed in the next section. 
Note that if $p>n$, then $\mu \in (W^{1,p}_{0}(\Omega))^{*} \subset (W^{1,H}_{0}(\Omega))^{*}$ and we are in the realm of classical weak solutions; in this case,  regularity results for weak solutions to \eqref{main.eq} are well established in the aforementioned references. 
Therefore, in this paper, it is not restrictive to consider the case $p \le n$ only, in which problem \eqref{main.eq} does not in general have weak solutions in the natural energy space $W^{1,H}_{0}(\Omega)$. Thus, a weaker notion of solutions should be considered in this case. 
Among various notions of solutions, we deal with the notion of SOLA introduced in \cite{BG89}, see the next section for details.   
We also refer to \cite{BBGGPV95,Ch18,CiMa17NA,DMOP99} for other notions of solutions. 

In this paper, we are interested in gradient integrability estimates of Calder\'on--Zygmund type. 
For nonlinear elliptic measure data problems with standard $p$-growth, modeled on
\begin{equation}\label{measure.p}
-\mathrm{div}\,(|Du|^{p-2}Du) = \mu,
\end{equation}
Mingione \cite{Min10} employed the $1$-fractional maximal function of $\mu$, defined by
\begin{equation}\label{def.M1} 
\mathbf{M}_{1}(\mu)(x) \coloneqq \sup_{r>0}\frac{|\mu|(B_{r}(x))}{r^{n-1}}, \qquad x \in \mathbb{R}^{n},
\end{equation}
to obtain local gradient estimates for \eqref{measure.p} in Lorentz--Morrey spaces. This approach was further developed by Phuc \cite{Phuc}, who obtained global weighted gradient estimates for \eqref{measure.p} and applied them to related Riccati type equations.

Later, such gradient estimates were extended to elliptic measure data problems with various types of nonstandard growth, such as Orlicz growth \cite{BCY21a}, $p(x)$-growth \cite{BOP17}, and mild phase transition \cite{BCY21b}. 
However, as far as we are aware, there is no analogous result for measure data problems with double phase growth. 
Indeed, the nonuniform ellipticity of the double phase operator, which is stronger than those of operators with $p(x)$-growth or mild phase transition, causes several difficulties in establishing regularity estimates. Moreover, in order to handle the  difficulties in the setting of measure data problems, one needs new and nontrivial ideas, see Section \ref{sec:tech}.

The aim of this paper is to prove local Calder\'on--Zygmund type estimates for the measure data problem \eqref{main.eq} with double phase growth. To this aim, we find a natural structural assumption (see \eqref{rate.sola} below), which is a suitable replacement of \eqref{rate.weak.sol}, and establish new reverse H\"older type estimates and comparison estimates. 
Then, with the help of these ingredients, we prove our main result given in Theorem \ref{main.thm} below by  applying the maximal function free technique introduced in \cite{AM} and revisited in \cite{BCY21a,BCY21b}.
To the best of our knowledge, our result is the first one concerning gradient regularity for double phase problems with measure data.

\subsection{Assumptions and main result}

Here we introduce our main assumptions on $p$, $q$, and $a(\cdot)$ for our regularity results. We assume that 
\begin{equation}\label{a.holder}
a \in C^{0,\alpha}(\Omega) \qquad \text{for some}\;\; \alpha \in (0,1]
\end{equation}
and that the growth exponents $p$ and $q$ satisfy
\begin{equation}\label{rate.sola}
\frac{q-1}{p-1} < 1+\frac{\alpha}{n-1}.
\end{equation}

\begin{remark}
Let us briefly discuss assumption \eqref{rate.sola}. 
\begin{itemize}
\item[(i)] In the papers \cite{BCM18,CM15b,CM16JFA,DM19} concerned with minimizers or weak solutions, assumption \eqref{rate.weak.sol} is used to correct the nonstandard growth of the double phase functional \eqref{dp.ftnal} with respect to the gradient variable. More precisely, in view of \eqref{ellipticity}, it allows one to control the second term in \eqref{dp.ftnal} in terms of $W^{1,p}$-energy when $a(x)$ is close to $0$,  thereby reducing several regularity estimates for \eqref{dp.ftnal} to those for the $p$-Dirichlet functional in such a situation. We point out that \eqref{rate.weak.sol} stems from the fact that local minimizers of \eqref{dp.ftnal} are $W^{1,p}$-regular.  

\item[(ii)] The role of our assumption \eqref{rate.sola} is essentially the same; the difference is due to the fact that solutions to measure data problems enjoy a modest regularity. In fact, since $\eqref{model}$ reduces to \eqref{measure.p} in the set $\{a(x) = 0\}$, we have that every SOLA to \eqref{main.eq} 
is $W^{1,s}$-regular for any $1 \le s < n(p-1)/(n-1)$, see Proposition \ref{prop.existence} below. 
In this situation, again recalling \eqref{ellipticity}, assumption \eqref{rate.sola} allows one to control the second term in \eqref{model} in terms of the $W^{1,s}$-norm of $Du$ when $a(x)$ is close to $0$. In this sense, our assumption \eqref{rate.sola} can be considered as a natural measure data analog of \eqref{rate.weak.sol}. 

\item[(iii)] Note in particular that when $p \le n$, \eqref{rate.sola} implies \eqref{rate.weak.sol}. 
\end{itemize}
\end{remark}

We now state our main theorem, see the next section for the definition of SOLA. 
Throughout this paper, we denote
\[ \data \coloneqq (n,p,q,\nu,L,\|a\|_{L^{\infty}},\alpha,[a]_{\alpha}, \|Du\|_{L^{\kappa}(\Omega)}), \]
where $\kappa$ is a constant given in \eqref{def.kappa} below.

\begin{theorem}\label{main.thm}
Let $u$ be a SOLA to \eqref{main.eq} under assumptions \eqref{growth}, \eqref{p.bound}, \eqref{a.holder} and \eqref{rate.sola}. 
Then we have the implication
\begin{equation} \label{cz.implication}
\mathbf{M}_{1}(\mu) \in L^{\gamma}_{\loc}(\Omega) \;\; \Longrightarrow \;\; h(\cdot,|Du|) \in L^{\gamma}_{\loc}(\Omega) \qquad \text{for any }\; \gamma \in (1,\infty). 
\end{equation}
Moreover, for any $\gamma \in (1,\infty)$, there exist a radius $R_{0} \equiv R_{0}(\data,|\mu|(\Omega),\gamma) \in (0,1)$ and constants $c\equiv c(\data) \ge 1$ and $c_{\gamma} \equiv c_{\gamma}(\data,\gamma) \ge 1$ such that
\begin{equation}\label{main.est}
\left(\mean{B_{R}}[h(x,|Du|)]^{\gamma}\,dx\right)^{1/\gamma} \le c\mean{B_{2R}}h(x,|Du|)\,dx  + c_{\gamma}\left(\mean{B_{2R}}\left[\mathbf{M}_{1}(\mu)\right]^{\gamma}\,dx\right)^{1/\gamma} 
\end{equation}
holds for any ball $B_{2R} \Subset \Omega$ with $2R \in (0,R_{0})$.
\end{theorem}

\subsection{Techniques and novelties}\label{sec:tech}
Here we outline our approach and point out the main novelties.

A crucial ingredient in our proof of Theorem \ref{main.thm} is a gradient integrability estimate for the homogeneous equation
\begin{equation}\label{homo.equation} 
\mathrm{div}\,A(x,Dw) = 0,
\end{equation}
whose proof is not trivial when $a(x)$ is close to $0$. 
In the literature concerned with weak solutions to \eqref{homo.equation}, the constants in such estimates usually depend on the $L^{p}$-norm of $Dw$, see for instance \cite{BO,CM15b,CM16JFA}. 
This is a common phenomenon in nonuniformly elliptic problems. However, such quantities are not applicable in our setting of measure data problems, and we need to obtain an additional reverse H\"older type estimate for \eqref{homo.equation} (see  Section \ref{sec3}), which is a new result. 
In the case of $p(x)$-growth \cite{BOP17} or mild phase transition \cite{BCY21b}, similar reverse H\"older type estimates were obtained by combining standard energy estimates and Sobolev--Poincar\'e inequalities. 
Such an approach cannot be applied directly to equation \eqref{homo.equation} with double phase growth, since the Sobolev--Poincar\'e type inequality for $w$ (see for instance \cite[Theorem 1.6]{CM15b}) still involves constants depending on the $L^p$-norm of $Dw$. 
We thus develop a new approach, which employs a weighted energy type estimate for $w$ and Sobolev--Poincar\'e inequalities for usual Sobolev spaces, along with \eqref{rate.sola}. 
This gives a preliminary version of reverse H\"older type estimate, which can be iterated finitely many times in order to yield the desired estimate. 
This is the point where the exponent $\kappa$ defined in \eqref{def.kappa} comes into play. 

Another important tool is a comparison estimate between \eqref{main.eq} and a Dirichlet problem involving \eqref{homo.equation}. 
In this paper, we also obtain a new comparison estimate below the natural energy space, which extends those concerned with $p$-Laplacian type equations established in \cite{KM14BMS,Min07} in a precise form. Note that, in contrast with those obtained in \cite[Step 2 in the proof of Theorem~1.1]{CM16JFA}, we prove our comparison estimate (see \eqref{comp1} below) by dividing cases into the $p$-phase (the case when $a(x)$ is close to $0$) and the $(p,q)$-phase (the case when $a(x)$ is far from $0$), with the help of \eqref{rate.sola}.
We believe that the comparison estimate obtained in this paper will be very useful in proving various kinds of gradient regularity results, such as gradient potential estimates and higher fractional differentiability, for \eqref{main.eq}.

\begin{remark} We finally discuss some possible extensions and open questions.
\begin{itemize}
\item[(i)] We expect that, with the same spirit as in \cite{BL22,BO}, it will be possible to establish global gradient estimates for problems of the type \eqref{main.eq}. 
\item[(ii)] We also expect that it will be possible to extend Theorem~\ref{main.thm} to the subquadratic case; however, this case needs somewhat different approaches and will be considered in a forthcoming work. Indeed, not only the comparison estimates, but also the proof of regularity estimates for homogeneous equations (see Section \ref{sec3} below) needs modification. 
\item[(iii)] Specifically, the so-called strongly singular case $1<p\le 2-1/n$ is much harder, in which one needs to consider other notions of solutions; moreover, such solutions (in any sense) to \eqref{main.eq} do not belong to the space $W^{1,1}(\Omega)$, which causes a substantial difficulty in establish comparison estimates. Gradient estimates for singular $p$-Laplacian type equations with measure data can be found in \cite{NP19,NP23ARMA}.
\item[(iv)] Similarly to the case of \eqref{measure.p} (see for instance \cite{BBGGPV95,Min07}), it is actually possible to show that any SOLA to \eqref{main.eq} belongs to the Marcinkiewicz space $L^{n(p-1)/(n-1),\infty}$. In this point of view, an interesting question is whether our regularity results continue to hold in the borderline case
\[ \frac{q-1}{p-1} = 1 + \frac{\alpha}{n-1}. \]
As usual, the borderline case exhibits several subtle issues; specifically,  the proofs of several estimates in this paper depend on whether $\alpha<1$ or $\alpha=1$ in this borderline case. 
\end{itemize}
\end{remark}

The remaining part of this paper is organized as follows. In the next section, we introduce basic notations and function spaces, and discuss the existence of SOLA to \eqref{main.eq}. 
In Section \ref{sec3}, we obtain reverse H\"older type estimates and gradient integrability results for homogeneous equations linked to \eqref{main.eq}. 
Finally, in Section \ref{sec4}, we establish suitable approximation and comparison estimates to prove Theorem \ref{main.thm} via exit time and covering arguments.

\section{Preliminaries} 
\subsection{Notation}
We denote by $c$ a general constant greater than or equal to one, 
whose value may vary from line to line. Specific dependencies of constants are denoted by using parentheses.

For any $p >1$, we denote its H\"older conjugate by $p' \coloneqq p/(p-1)$. Moreover,
\begin{equation*}
p^{*} \coloneqq 
\left\{
\begin{aligned}
&\frac{np}{n-p} & \text{if }& p < n, \\
&\text{any number in }(p,\infty) & \text{if }& p \ge n
\end{aligned}
\right.
\end{equation*}
and 
\[ p_{*} \coloneqq \max\left\{\frac{np}{n+p},1\right\} \]
denotes the Sobolev conjugate and the inverse Sobolev conjugate of $p$, respectively.

As usual, we denote by $B_{r}(x_0) \coloneqq \{x\in\mathbb{R}^{n}:|x-x_0|<r\}$ the $n$-dimensional open ball with center $x_0 \in \mathbb{R}^n$ and radius $r>0$. If there is no confusion, we omit the center and simply denote $B_{r}\equiv B_{r}(x_0)$. 
Also, given a ball $B$, we denote by $\gamma B$ the concentric ball with radius magnified by a factor $\gamma>0$. 
Unless otherwise stated, different balls in the same context are concentric. 

The ($n$-dimensional) Lebesgue measure of a measurable set $U \subset \mathbb{R}^n$ is denoted by $|U|$. For an integrable map $f:U \to \mathbb{R}^{k}$, with $k \in \mathbb{N}$ and $0<|U|<\infty$, we denote the integral average of $f$ over $U$ by
\begin{equation*}
(f)_{U} \coloneqq \mean{U}f\,dx  \coloneqq \frac{1}{|U|} \int_{U}f\,dx .
\end{equation*}

\subsection{Function spaces}
Observe that the ellipticity assumption $\eqref{growth}_{2}$ on $A(\cdot)$ implies the following monotonicity condition:
\begin{equation}\label{mono}
(|z_{1}|+|z_{2}|)^{p-2}|z_{1}-z_{2}|^{2} + a(x)(|z_{1}|+|z_{2}|)^{q-2}|z_{1}-z_{2}|^{2} \le c(A(x,z_{1})-A(x,z_{2}))\cdot(z_{1}-z_{2})
\end{equation}
for any $x \in \Omega$ and $z_{1},z_{2} \in \mathbb{R}^{n}$, where $c\equiv c(n,p,q,L,\nu)$.
In addition, when $2 \le p <q$, we further have
\begin{equation}\label{mono2}
|z_{1}-z_{2}|^{p} + a(x)|z_{1}-z_{2}|^{q} \le c(A(x,z_{1})-A(x,z_{2}))\cdot(z_{1}-z_{2}).
\end{equation}

We recall the function $H(\cdot)$ defined in \eqref{def.H}. It is a generalized Young functions (see \cite{HH19book}) that satisfies the $\Delta_2$ and $\nabla_2$ conditions: $H(x,2t) \le 2^q H(x,t)$ and $H(x,t) \le 2^{-p/(p-1)}H(x,2^{1/(p-1)}t)$ for any $x \in \Omega$ and $t \ge 0$. Accordingly, for any open set $U \subseteq \Omega$,
we consider the Musielak--Orlicz space $L^{H}(U)$, which is defined as the set of all measurable function $f:U \to \mathbb{R}^n$ such that
\[ \int_{U}H(x,|f|)\,dx < \infty. \]
Due to the $\Delta_2$ and $\nabla_2$ conditions, it is a separable, reflexive Banach space endowed with the Luxemburg norm
\[ \|f\|_{L^{H}(U)} \coloneqq \inf\left\{ \lambda>0: \int_{U}H(|f|/\lambda)\,dx < \infty \right\}. \] 
Moreover, the Musielak--Orlicz--Sobolev space $W^{1,H}(U)$ is defined as the set of all functions $f \in L^{H}(U) \cap W^{1,1}(U)$ satisfying
\[ \int_{U}H(x,|Df|)\,dx < \infty, \]
which is also a separable reflexive Banach space endowed with the norm
\[ \|f\|_{W^{1,H}(U)} \coloneqq \|f\|_{L^H(U)} + \|Df\|_{L^H(U)}. \]
We also denote by $W^{1,H}_{0}(U)$ the closure of $C^{\infty}_{0}(U)$ in $W^{1,H}(U)$. 
For more on generalized Young functions and Musielak--Orlicz spaces, see \cite{HH19book}. 

\subsection{SOLA (Solutions Obtained as Limits of Approximations)}

Here we discuss the existence of SOLA to \eqref{main.eq}. 
Note that in this subsection, we consider the following sharp range of the growth exponents:
\begin{equation}\label{sharp.range}
2-1/n < p \le n \qquad \text{and} \qquad p < q < \infty. 
\end{equation}

\begin{definition}
Under assumptions \eqref{growth} and \eqref{sharp.range}, we say that $u \in W^{1,1}_{0}(\Omega)$ is a SOLA to \eqref{main.eq} if it is a distributional solution to $\eqref{main.eq}_1$, i.e., $A(x,Du) \in L^{1}(\Omega)$ and
\[ \int_{\Omega}A(x,Du)\cdot D\varphi\,dx = \int_{\Omega}\varphi \,d\mu \]
holds for any $\varphi \in C^{\infty}_{0}(\Omega)$. 
Moreover, it has to satisfy the following approximation property: there exists a sequence of functions $\{\mu_{j}\}\subset L^{\infty}(\Omega)$ such that $\mu_{j}\rightharpoonup \mu$ weakly* in the sense of measures and satisfies
\begin{equation}\label{muj.conv}
 \limsup_{j}|\mu_j|(B) \le |\mu|(\bar{B}) \quad \text{for any ball } B \subset \mathbb{R}^{n},
\end{equation}
and there exists a sequence of weak solutions $\{u_{j}\} \subset W^{1,H}_{0}(\Omega)$ to
\begin{equation}\label{approx.prob}
\left\{
\begin{aligned}
-\text{div}\,A(x,Du_{j}) & = \mu_{j}& \text{in }& \Omega, \\
u_{j} & = 0& \text{on }& \partial\Omega
\end{aligned}
\right.
\end{equation}
such that $u_{j} \to u$ in $W^{1,1}_{0}(\Omega)$.
\end{definition}

In order to discuss the existence of SOLA,  we first establish a global a priori estimate for \eqref{main.eq}, whose proof is similar to that of  \cite[Lemma~4.4]{BSS22}.

\begin{lemma}\label{global.est}
Let $u\in W^{1,H}_{0}(\Omega)$ be the weak solution to \eqref{main.eq} under assumptions \eqref{growth} and \eqref{sharp.range}. Then for any $s$ satisfying
\begin{equation}\label{exp.range}
1 \le s < \frac{n(p-1)}{n-1},
\end{equation}
we have
\begin{equation}\label{globalest}
\int_{\Omega}|Du|^{s}\,dx \leq c[|\mu|(\Omega)]^{s/(p-1)}
\end{equation}
for a constant $c\equiv c(n,p,q,\nu,L,|\Omega|,s)$. 
\end{lemma}

\begin{proof}
We first obtain a weighted type energy estimate. 
For numbers $d>0$ and $\xi >1$, we set 
$\eta_{\pm} \coloneqq d^{1-\xi} - (d+u_{\pm})^{1-\xi}$.
Testing \eqref{main.eq} with $\eta_{\pm} \in W^{1,H}_{0}(\Omega) \cap L^{\infty}(\Omega)$, we obtain
\begin{equation*}
|I_{\pm}|  \coloneqq \left|(\xi-1)\int_{\Omega}\frac{ A(x,Du) \cdot Du_{\pm} }{(d + u_{\pm})^{\xi}}\,dx \right| 
 = \left|\int_{\Omega}\eta_{\pm}\,d\mu \right| 
 \leq cd^{1-\xi}|\mu|(\Omega).
\end{equation*}
In light of \eqref{mono}, we have
\[ \int_{\Omega}\frac{|Du|^{p}}{(d+|u|)^{\xi}}\,dx \leq \frac{c}{\xi-1}(|I_{+}|+|I_{-}|) \leq c\frac{d^{1-\xi}}{\xi-1}|\mu|(\Omega) \]
for a constant $c\equiv c(n,p,q,\nu,L)$. 
Now, in the above display, we choose 
\[ \xi = \frac{n(p-s)}{n-s} \;\; \Longleftrightarrow \;\; \frac{\xi s}{p-s} = \frac{ns}{n-s} = s^{*} , \]
which is admissible due to \eqref{exp.range}, and 
\[ d = \left(\int_{\Omega}|u|^{s^{*}}\,dx \right)^{1/s^*} \leq c\left(\int_{\Omega}|Du|^{s}\,dx\right)^{1/s}. \]
We may assume that $d>0$, otherwise $u \equiv 0$ in $\Omega$ and there is nothing to prove. We then proceed as 
\begin{equation*}
\begin{aligned}
\int_{\Omega}|Du|^{s}\,dx 
& \leq \left( \int_{\Omega}\frac{|Du|^{p}}{(d+|u|)^{\xi}}\,dx\right)^{s/p}
\left( \int_{\Omega}(d+|u|)^{s^{*}}\,dx\right)^{(p-s)/p} \\
& \leq c\left(|\mu|(\Omega)d^{1-\xi}\right)^{s/p} d^{\xi s/p} \\
& \leq c[|\mu|(\Omega)]^{s/p}\left(\int_{\Omega}|Du|^{s}\,dx\right)^{1/p} \\
& \leq c[|\mu|(\Omega)]^{s/(p-1)} + \frac{1}{2}\int_{\Omega}|Du|^{s}\,dx.
\end{aligned}
\end{equation*}
Reabsorbing the last term in the right-hand side into the left-hand side, we obtain \eqref{globalest}.
\end{proof}

\begin{proposition}\label{prop.existence}
Under assumptions \eqref{growth}, \eqref{a.holder}, \eqref{rate.sola} and \eqref{sharp.range}, there exists a SOLA $u$ to \eqref{main.eq} satisfying $u \in W^{1,s}_{0}(\Omega)$ for any $s$ satisfying \eqref{exp.range}.
\end{proposition}
\begin{proof}
We set $\mu_{j} \coloneqq \mu \ast \phi_{j}$ for each $j \in \mathbb{N}$, with $\{\phi_{j}\}$ being a sequence of standard mollifiers. Then the sequence $\{\mu_{j}\} \subset L^{\infty}(\Omega)$ converges to $\mu$ weakly* in the sense of measures and satisfies \eqref{muj.conv} as well as
\[ |\mu_j|(\Omega) \le |\mu|(\Omega). \] 
We accordingly consider the sequence $\{u_{j}\}$ of weak solutions to \eqref{approx.prob}; the existence of such $u_{j}$ can be proved via a monotonicity method in Musielak--Orlicz spaces (see for instance \cite[Lemma~3.4]{BL22}). 
Applying Lemma~\ref{global.est} to each $u_{j}$, we see that $\{u_{j}\}$ is bounded in $W^{1,s}_{0}(\Omega)$ for any $s$ satisfying \eqref{exp.range}. Accordingly, there exists a function $u \in W^{1,1}_{0}(\Omega)$ such that $u_{j} \rightharpoonup u$ in $W^{1,s}_{0}(\Omega)$ and $u_{j} \to u$ in $L^{s}(\Omega)$ (up to a subsequence). 
Then, by using the truncation and compactness methods in \cite[Theorem~1]{BG89}, we have that $Du_{j} \to Du$ in $L^{s}(\Omega)$ and a.e. in $\Omega$. Finally, using the fact that
\begin{equation}\label{q-1.admissible}
\eqref{rate.sola}, \, \alpha \le 1 \;\; \Longrightarrow \;\; \frac{q-1}{p-1} < 1+\frac{1}{n-1} \;\; \Longleftrightarrow \;\; q-1 < \frac{n(p-1)}{n-1}, 
\end{equation}
along with Vitali's convergence theorem, we conclude that $A(x,Du_{j}) \to A(x,Du)$ in $L^{1}(\Omega)$ and $u$ is a distributional solution to $\eqref{main.eq}_1$. This completes the proof. 
\end{proof}

\section{Regularity for reference problems}\label{sec3}
In this section, we investigate several regularity results for homogeneous equations related to \eqref{main.eq}. 
We fix a ball $B_{r} \Subset \Omega$ with $r \le 1$,  and consider 
\begin{equation}\label{homogeneous}
-\mathrm{div}\,A(x,Dw) = 0 \quad \text{in}\;\; B_{r}
\end{equation}

\subsection{Reverse H\"older estimates in the $p$-phase}

In this subsection, we establish reverse H\"older estimates for \eqref{homogeneous} under the assumption
\begin{equation}\label{p.phase.4r}
\sup_{x \in B_{r}}a(x) \le \bar{L}[a]_{\alpha}r^{\alpha}
\end{equation}
for some $\bar{L} \in [8,\infty)$.

Here, we set
\begin{equation}\label{def.kappa}
\kappa \coloneqq \frac{1}{2}\left[ \max\left\{ \frac{n(q-p)}{\alpha}, \frac{n(n+\alpha)(p-1)}{n(n-1) + (n+\alpha)(p-1)} \right\} + \frac{n(p-1)}{n-1} \right].
\end{equation}
Since $p \geq 2$, \eqref{def.kappa} directly implies
\[ \kappa \geq \frac{1}{2} \left[ \frac{n(p-1)}{n+p-2} + \frac{n(p-1)}{n-1} \right] \geq \frac{1}{2} \left[ 1 + \frac{n}{n-1} \right] > 1. \]
Moreover, in light of the facts that
\begin{equation*}
\eqref{rate.sola} \;\; \Longleftrightarrow \;\; \frac{n(q-p)}{\alpha} < \frac{n(p-1)}{n-1} 
\end{equation*}
and that
\begin{equation*}
\left( \frac{n(n+\alpha)(p-1)}{n(n-1) + (n+\alpha)(p-1)} < \frac{n(p-1)}{n-1} \right),
\end{equation*}
we have
\begin{equation}\label{kappa.range} 
\kappa \in \left( \max\left\{ \frac{n(q-p)}{\alpha}, \frac{n(n+\alpha)(p-1)}{n(n-1) + (n+\alpha)(p-1)} \right\} ,\frac{n(p-1)}{n-1}\right). 
\end{equation}

We now define a continuous injective function $m: [n(p-1)/(n-1), p] \to \mathbb{R}$ by 
\begin{align}\label{def.m}
m(\bar p) = (\bar p + (q-p))_{*} \qquad \text{for }\; \bar{p} \in [n(p-1)/(n-1),p] \end{align}
and set $ \xi \coloneqq p - \bar{p} \in [0, (n-p)/(n-1)]$ for $\bar{p} \in [n(p-1)/(n-1),p]$. Observe that, since $2 \leq p \le n$, 
\begin{equation*}
\begin{aligned}
1+\xi < 2 \le p \;\; & \Longrightarrow \;\; \frac{n-p+\xi}{n-1} < 1 \le p-1 < \frac{(p-\xi)^{2}}{p-1} \\
& \Longrightarrow \;\; \frac{(p-1)(n-p+\xi)}{n-1} < (p-\xi)^{2} \\
& \Longrightarrow \;\; np-p^{2}+p\xi + \frac{(p-1)(n-p+\xi)}{n-1} < np-p\xi+\xi^{2} \\
\overset{\eqref{rate.sola}}&{\Longrightarrow} \;\; q < p+\frac{\alpha(p-1)}{n-1} \le p + \frac{p-1}{n-1} < \frac{np-p\xi+\xi^{2}}{n-p+\xi} \\
& \Longrightarrow \;\; q < \frac{np-p\xi+\xi^{2}}{n-p+\xi} = \xi + \frac{np-n\xi}{n-p+\xi} = \xi + (p-\xi)^{*} 
\end{aligned}
\end{equation*}
for any $\xi = p - \bar{p} \in [0,(n-p)/(n-1)]$. In other words, 
\begin{align}\label{m.cont}
m(\bar p) < \bar p \qquad \text{for any}\;\; \bar p \in [n(p-1)/(n-1), p].
\end{align}
Also note that
\begin{equation}\label{m.bd}
\begin{aligned}
m \left( \frac{n(p-1)}{n-1} \right) \overset{\eqref{rate.sola}}{<} \left( \frac{(n+\alpha)(p-1)}{n-1} \right)_{*} = \frac{n(n+\alpha)(p-1)}{n(n-1) + (n+\alpha)(p-1)} \overset{\eqref{kappa.range}} {<} \kappa.
\end{aligned}
\end{equation}
Due to the fact that $\bar{p} \mapsto m(\bar p)/\bar p$ is continuous on $[n(p-1)/(n-1),p]$, along with \eqref{m.cont} and \eqref{m.bd}, we can choose a number $N_{0} \equiv N_{0}(n,p,q,\alpha) \in \mathbb{N} \cup \{ 0 \}$ such that
\[ m^{N_{0}+1}(p) \leq \kappa < m^{N_{0}}(p). \]

We first obtain a preliminary version of reverse H\"older type estimates.

\begin{lemma}\label{lem.self.1}
Let $w$ be a weak solution to \eqref{homogeneous} under assumptions \eqref{growth}, \eqref{p.bound}, \eqref{a.holder}, and \eqref{rate.sola}. 
Assume that $B_r$ satisfies \eqref{p.phase.4r}  for some $\bar{L} \in [8 , \infty)$.
For any $1/2 \le \theta_{2} < \theta_{1} \le 1$ and $\bar{p} \in [n(p-1)/(n-1), p]$, we have
\begin{equation*}
\begin{aligned}
 \left( \mean{B_{\theta_{2}r}}|Dw|^{\bar p}\,dx \right)^{1/\bar p}
& \le c \left[ 1+ r^{\alpha} \left(\mean{B_{\theta_{1}r}}|Dw|^{m(\bar p)}\,dx \right)^{(q-p)/m(\bar p)} \right] \left(\mean{B_{\theta_{1}r}}|Dw|^{m(\bar p)}\,dx \right)^{1/m(\bar p)}
\end{aligned}
\end{equation*}
for a constant $c\equiv c(n,p,q,\nu,L, \|a\|_{L^{\infty}}, \alpha, [a]_{\alpha}, \bar p, \bar{L}, \theta_{1}-\theta_{2})$.
\end{lemma}

\begin{proof}
In this proof, $c$ depends on $n, p, q, \nu, L, \|a\|_{L^{\infty}},\alpha, [a]_{\alpha}, \bar p, \bar{L}$ and $\theta_{1}-\theta_{2}$, where we omit writing the dependency.
Let $\eta \in C^{\infty}_{0}(B_{\theta_{1}r})$ be a cut-off function satisfying $\eta \equiv 1$ in $B_{\theta_{2} r}$ and $|D\eta| \le 4/[(\theta_{1} - \theta_{2})r]$.
We simply write $\xi = p - \bar p \in [0, 1)$, $B_{\rho}^{+} \coloneqq \{ x \in B_{\rho} : w(x) - (w)_{B_{\theta_{1}r}} \geq 0 \}$, and $B_{\rho}^{-} \coloneqq B_{\rho} \setminus B_{\rho}^{+}$ for any $\rho \in (0, \theta_{1} r]$.
Testing \eqref{homogeneous} with $\varphi \equiv (d+(w-(w)_{B_{\theta_{1} r}})_{+})^{1-\xi}\eta^{q}$ for any $d>0$, we have
\begin{equation*}
\begin{aligned}
& \int_{B^{+}_{\theta_{1} r}}\left[\frac{|Dw|}{(d+(w-(w)_{B_{\theta_{1} r}})_{+})^{\xi/p}}\right]^{p}\eta^{q} + a(x)\left[\frac{|Dw|}{(d+(w-(w)_{B_{\theta_{1} r}})_{+})^{\xi/q}}\right]^{q}\eta^{q}\,dx \\
& \le c \int_{B_{\theta_{1} r}} ( |Dw|^{p-1} + a(x) |Dw|^{q-1}) \frac{(d + (w - (w)_{B_{\theta_{1}r}})_{+})^{1 - \xi}}{(1 - \xi)(\theta_{1} - \theta_{2}) r} \eta^{q-1} \, dx \\
& \le c \frac{d^{1- \xi}}{(1 - \xi)(\theta_{1} - \theta_{2}) r} \int_{B^{-}_{\theta_{1} r}} |Dw|^{p-1} + a(x) |Dw|^{q-1} \dx \\
& \quad + c \int_{B^{+}_{\theta_{1} r}} \Bigg\{ \left[ \frac{|Dw|}{(d + (w - (w)_{B_{\theta_{1} r}})_{+})^{\xi/p}} \right]^{p-1} \frac{(d+(w-(w)_{B})_{+})^{1-\xi/p}}{(1-\xi)(\theta_{1} - \theta_{2})r} \eta^{q-1} \\
& \hspace{2.5cm} + a(x) \left[ \frac{|Dw|}{(d + (w - (w)_{B_{\theta_{1} r}})_{+})^{\xi/q}} \right]^{q - 1} \frac{(d+(w-(w)_{B})_{+})^{1-\xi/q}}{(1-\xi)(\theta_{1} - \theta_{2})r} \eta^{q-1} \Bigg\} \, dx.
\end{aligned}
\end{equation*}
Then Young's inequality gives
\begin{equation*}
\begin{aligned}
& \int_{B^{+}_{\theta_{1} r}}\left[\frac{|Dw|}{(d+(w-(w)_{B_{\theta_{1} r}})_{+})^{\xi/p}}\right]^{p}\eta^{q} + a(x)\left[\frac{|Dw|}{(d+(w-(w)_{B_{\theta_{1} r}})_{+})^{\xi/q}}\right]^{q}\eta^{q}\,dx \\
& \le c \int_{B^{+}_{\theta_{1} r}}\frac{(d+(w-(w)_{B_{\theta_{1}r}})_{+})^{p-\xi}}{[(1-\xi)(\theta_{1} - \theta_{2})r]^{p}} + a(x)\frac{(d+(w-(w)_{B_{\theta_{1} r}})_{+})^{q-\xi}}{[(1-\xi)(\theta_{1} - \theta_{2})r]^{q}}\,dx \\
& \quad + c \frac{d^{1- \xi}}{(1 - \xi)(\theta_{1} - \theta_{2}) r} \int_{B^{-}_{\theta_{1} r}} |Dw|^{p-1} + a(x) |Dw|^{q-1} \dx.
\end{aligned}
\end{equation*}
In a completely similar way, this time testing \eqref{homogeneous} with $\varphi \equiv (d+(w-(w)_{B_{\theta_{1} r}})_{-})^{1-\xi}\eta^{q}$, we also have
\begin{equation*}
\begin{aligned}
& \int_{B^{-}_{\theta_{1} r}}\left[\frac{|Dw|}{(d+(w-(w)_{B_{\theta_{1} r}})_{-})^{\xi/p}}\right]^{p}\eta^{q} + a(x)\left[\frac{|Dw|}{(d+(w-(w)_{B_{\theta_{1} r}})_{-})^{\xi/q}}\right]^{q}\eta^{q}\,dx \\
& \le c \int_{B^{-}_{\theta_{1} r}}\frac{(d+(w-(w)_{B_{\theta_{1}r}})_{-})^{p-\xi}}{[(1-\xi)(\theta_{1} - \theta_{2})r]^{p}} + a(x)\frac{(d+(w-(w)_{B_{\theta_{1} r}})_{-})^{q-\xi}}{[(1-\xi)(\theta_{1} - \theta_{2})r]^{q}}\,dx \\
& \quad + c \frac{d^{1- \xi}}{(1 - \xi)(\theta_{1} - \theta_{2}) r} \int_{B^{+}_{\theta_{1} r}} |Dw|^{p-1} + a(x) |Dw|^{q-1} \dx.
\end{aligned}
\end{equation*}
Combining the last two displays, we obtain
\begin{equation*}
\begin{aligned}
& \int_{B_{\theta_{2} r}} \left[\frac{|Dw|}{[r^{-1}(d+|w-(w)_{B_{\theta_{1} r}}|)]^{\xi/p}}\right]^{p}\,dx \\
& \le r^{\xi}\int_{B_{\theta_{2} r}} \left[\frac{|Dw|}{(d+|w-(w)_{B_{\theta_{1}r}}|)^{\xi/p}}\right]^{p} + a(x)\left[\frac{|Dw|}{(d+|w-(w)_{B_{\theta_{1}r}}|)^{\xi/q}}\right]^{q}\,dx \\
& \le \frac{c}{(1-\xi)^{q}(\theta_{1}-\theta_{2})^{q}}\int_{B_{\theta_{1} r}}\left[\frac{d+|w-(w)_{B_{\theta_{1} r}}|}{r}\right]^{p-\xi} + a(x)\left[\frac{d+|w-(w)_{B_{\theta_{1}r}}|}{r}\right]^{q-\xi}\,dx \\
& \quad + \frac{c}{(1-\xi)(\theta_{1}-\theta_{2})}\left(\frac{d}{r}\right)^{1-\xi}\int_{B_{\theta_{1} r}}|Dw|^{p-1}+a(x)|Dw|^{q-1}\,dx,
\end{aligned}
\end{equation*}
which in turn gives
\begin{equation*}
\begin{aligned}
& \int_{B_{\theta_{2} r}} |Dw|^{p-\xi}\,dx 
 = \int_{B_{\theta_{2} r}} \left[\frac{|Dw|}{[r^{-1}(d+|w-(w)_{B_{\theta_{1} r}}|)]^{\xi/p}} \right]^{p-\xi}\left[ \frac{(d+|w-(w)_{B_{\theta_{1} r}}|)}{r} \right]^{\xi(p-\xi)/p}\,dx \\
& \le \int_{B_{\theta_{2} r}}\left[ \frac{|Dw|}{r^{-1}(d+|w-(w)_{B_{\theta_{1} r}}|)^{\xi/p}}\right]^{p} + \left[ \frac{d+|w-(w)_{B_{\theta_{1} r}}|}{r}\right]^{p-\xi}\,dx \\
& \le \frac{c}{(1-\xi)^{q}(\theta_{1}-\theta_{2})^{q}}\int_{B_{\theta_{1} r}} \left[ \frac{d+|w-(w)_{B_{\theta_{1} r}}|}{r} \right]^{p-\xi} + a(x) \left[ \frac{d+|w-(w)_{B_{\theta_{1}r}}|}{r} \right]^{q-\xi}\,dx \\
& \quad + \frac{c}{(1-\xi)(\theta_{1}-\theta_{2})}\left(\frac{d}{r}\right)^{1-\xi}\int_{B_{\theta_{1} r}}|Dw|^{p-1}+a(x)|Dw|^{q-1}\,dx.
\end{aligned}
\end{equation*}
Then, by using \eqref{p.phase.4r}, we have
\begin{equation*}
\begin{aligned}
\int_{B_{\theta_2}r}|Dw|^{p-\xi}\,dx 
& \le \frac{c}{(1-\xi)^{q}(\theta_{1}-\theta_{2})^{q}}\int_{B_{\theta_{1} r}} \left[ \frac{|w-(w)_{B_{\theta_{1} r}}|}{r}\right]^{p-\xi} + r^{\alpha} \left[\frac{|w-(w)_{B_{\theta_{1} r}}|}{r}\right]^{q-\xi}\,dx \\
& \quad + \frac{c}{(1-\xi)(\theta_{1}-\theta_{2})}\left(\frac{d}{r}\right)^{1-\xi}\int_{B_{\theta_{1}r}}|Dw|^{p-1}+a(x)|Dw|^{q-1}\,dx \\
& \quad + \frac{c}{(1-\xi)^{q}(\theta_{1}-\theta_{2})^{q}}\left[ \left(\frac{d}{r}\right)^{p-\xi} + r^{\alpha} \left(\frac{d}{r}\right)^{q-\xi} \right],
\end{aligned}
\end{equation*}
Letting $d\to0$, taking averages and then applying the usual Sobolev--Poincar\'e inequalities, we arrive at
\begin{equation*}
\begin{aligned}
& \mean{B_{\theta_{2}r}}|Dw|^{p-\xi}\,dx
 \le \frac{c}{(1- \xi)^{q}(\theta_{1}-\theta_{2})^{q}} \mean{B_{\theta_{1} r}} \left[ \frac{|w-(w)_{B_{\theta_{1} r}}|}{r} \right]^{p-\xi} + r^{\alpha} \left[ \frac{|w-(w)_{B_{\theta_{1} r}}|}{r} \right]^{q-\xi} \,dx \\
& \le \frac{c}{(1-\xi)^{q}(\theta_{1}-\theta_{2})^{q}}\left[ \left(\mean{B_{\theta_{1}r}}|Dw|^{(p-\xi)_{*}}\,dx \right)^{(p-\xi)/(p-\xi)_{*}} + r^{\alpha} \left(\mean{B_{\theta_{1}r}}|Dw|^{m(\bar p)}\,dx \right)^{(q-\xi)/m(\bar p)} \right],
\end{aligned}
\end{equation*}
where we have also used the fact that $\theta_1 / \theta_2 \le 2$. 
Then an application of H\"older's inequality completes the proof.
\end{proof}

\begin{lemma}\label{lem.selfaux}
For $\kappa$ given in \eqref{def.kappa}, any fixed $1/2 \leq \theta_{3} < \theta_{2} < \theta_{1} \leq 1$, $\kappa \leq p_{1} < p_{2} < p_{3} \leq p$, and $r \in (0, 1]$, assume that $w \in W^{1,p}(B_{r})$ satisfies
\begin{align}\label{lem.selfaux.1}
\left( \mean{B_{\theta_{3} r}} |Dw|^{p_{3}} \, dx \right)^{1/p_{3}} \leq \bar c \left[ 1+ r^{\alpha} \left( \mean{B_{\theta_{2} r}} |Dw|^{p_{2}} \, dx \right)^{(q-p)/p_{2}} \right] \left( \mean{B_{\theta_{2} r}} |Dw|^{p_{2}} \, dx \right)^{1/p_{2}}
\end{align}
and
\begin{align}\label{lem.selfaux.2}
\left( \mean{B_{\theta_{2} r}} |Dw|^{p_{2}} \, dx \right)^{1/p_{2}} \leq M \left( \mean{B_{\theta_{1} r}} |Dw|^{p_{1}} \, dx \right)^{1/p_{1}},
\end{align}
where $\bar c$ and $M$ are positive constants.
Then we have
\begin{align*}
\left( \mean{B_{\theta_{3} r}} |Dw|^{p_{3}} \, dx \right)^{1/p_{3}} \leq M_{\mathrm{iter}} \left( \mean{B_{\theta_{1} r}} |Dw|^{p_{1}} \, dx \right)^{1/p_{1}},
\end{align*}
where $M_{\mathrm{iter}}$ is a constant depending only on $\bar c, M, n, p, q,  \|Dw\|_{L^{p_1}(B_{r})}$ and is an increasing function with respect to $M$ and $\| Dw \|_{L^{p_{1}}(B_{r})}$.
\end{lemma}

\begin{proof}
Combining \eqref{lem.selfaux.1} and \eqref{lem.selfaux.2}, we obtain
\begin{equation*}
\begin{aligned}
\left( \mean{B_{\theta_{3} r}} |Dw|^{p_{3}} \, dx \right)^{1/p_{3}} 
& \leq \bar c M \left[ 1+ M^{q-p} r^{\alpha} \left( \mean{B_{\theta_{1} r}} |Dw|^{p_{1}} \, dx \right)^{(q-p)/p_{1}} \right] \left( \mean{B_{\theta_{1} r}} |Dw|^{p_{1}} \, dx \right)^{1/p_{1}} \\
& \leq \bar c M \left( 1+ M^{q-p} r^{\alpha - n(q-p)/p_{1}} \|Dw\|_{L^{p_{1}}(B_{r})} \right) \left( \mean{B_{\theta_{1} r}} |Dw|^{p_{1}} \, dx \right)^{1/p_{1}} \\
& \eqqcolon M_{\mathrm{iter}}\left( \mean{B_{\theta_{1} r}} |Dw|^{p_{1}} \, dx \right)^{1/p_{1}}.
\end{aligned}
\end{equation*}
Since $ n (q-p)/\alpha < \kappa \le p_{1}$, we have $\alpha - n(q-p)/p_{1}>0$.
Hence, $M_{\mathrm{iter}}$ is an increasing function with respect to $M$ and $\| Dw \|_{L^{p_{1}}(B_{r})}$.
\end{proof}

The next is a refined version of reverse H\"older type inequalities.
\begin{lemma}
Let $w$ be a weak solution to \eqref{homogeneous} under assumptions \eqref{growth}, \eqref{p.bound}, \eqref{a.holder} and \eqref{rate.sola}. 
Assume that $B_{r}$ satisfies \eqref{p.phase.4r} for some $\bar{L} \in [8, \infty)$.
Then for the constant $\kappa$ given in \eqref{def.kappa}, we have
\begin{equation}\label{Lp.Lkappa}
\left( \mean{B_{r/2}}|Dw|^{p}\,dx \right)^{1/p}
\le c \left(\mean{B_{r}}|Dw|^{\kappa}\,dx \right)^{1/\kappa}.
\end{equation}
for a constant $c\equiv c(n,p,q,\nu,L, \|a\|_{L^{\infty}}, \alpha, [a]_{\alpha}, \bar{L}, \|Dw\|_{L^{\kappa}(B_{r})})$.
\end{lemma}

\begin{proof}
Recall the definition of  $m(\cdot)$ given in \eqref{def.m} and the properties given below the definition.
If $N_{0}=0$, then $m(p) \leq \kappa$ and so, Lemma~\ref{lem.self.1} directly gives the desired estimate.

We now assume $N_{0}>0$. 
Denote $\bar{\theta}_{k} = (2 N_{0} + 3 -k)/[2(N_{0} +1)]$ for $k = 1, 2, \dots, N_{0} + 2$ and
\[ \bar{p}_{1} = \kappa, \quad \bar{p}_{2} = m^{N_{0}}(p), \quad \bar{p}_{3} = m^{N_{0}-1}(p), \quad \dots, \quad \bar{p}_{N_{0}+2} = p. \]
Note in particular that $\bar{\theta}_{k} - \bar{\theta}_{k+1} = 1/[2(N_{0}+1)]$ depends only on $n$, $p$, $q$ and $\alpha$.
Applying Lemma~\ref{lem.self.1} with $\bar p = \bar{p}_{2}$ and Jensen's inequality, we have
\begin{align*}
\left( \mean{B_{\bar{\theta}_{2} r}} |Dw|^{\bar{p}_{2}} \, dx \right)^{1/ \bar{p}_{2}}
& \leq c \left( 1 + r^{\alpha - n(q-p)/\bar{p}_{1}} \| Dw \|_{L^{\bar{p}_{1}}(B_{r})} \right) \left( \mean{B_{r}} |Dw|^{\bar{p}_{1}} \, dx \right)^{1/\bar{p_{1}}}.
\end{align*}
Note that $\alpha - n(q-p)/\bar{p}_{1} = \alpha - n(q-p)/\kappa>0$.
Hence, there is $M = M(\data, \bar L, \| Dw \|_{L^{\kappa}(B_{r})})$, which is an increasing function of $\| Dw \|_{L^{\kappa}(B_{r})}$, such that
\begin{align*}
\left( \mean{B_{\bar{\theta}_{2} r}} |Dw|^{\bar{p}_{2}} \, dx \right)^{1/\bar{p}_{2}}
& \leq M \left( \mean{B_{r}} |Dw|^{\bar{p_{1}}} \, dx \right)^{1/\bar{p}_{1}}.
\end{align*}
By applying $N_0$-times of Lemma~\ref{lem.selfaux} with
\[ (\theta_{1}, \theta_{2}, \theta_{3}, p_{1}, p_{2}, p_{3}) = (1, \bar{\theta}_{k+1}, \bar{\theta}_{k+2} ,\kappa, \bar{p}_{k+1}, \bar{p}_{k+2}) \]
for $k=1,2,3, \dots , N_{0}$ in an increasing order of $k$, we achieve the desired result.
\end{proof}

With the help of the above lemma, we are now able to modify the gradient integrability result in \cite[Theorem~5.1]{CM15b} as follows.
\begin{lemma}\label{lem:higher}
Let $w$ be a weak solution to \eqref{homogeneous} under assumptions \eqref{growth}, \eqref{p.bound}, \eqref{a.holder} and \eqref{rate.sola}. Assume that $B_{r}$ satisfies \eqref{p.phase.4r} for some $\bar{L} \in [8,\infty)$. Then for any $\tilde{q} < np/(n-2\alpha)$ ($= \infty$ when $n=2$ and $\alpha=1$), there exists a constant $c \equiv c(n,p,q,\nu,L, \|a\|_{L^{\infty}}, \alpha, [a]_{\alpha},\bar{L},\|Dw\|_{L^{\kappa}(B_{r})},\tilde{q})$ such that 
\begin{equation}\label{higher.est}
\left(\mean{B_{r/2}}|Dw|^{\tilde{q}}\,dx\right)^{1/\tilde{q}} \le c\left(\mean{B_{r}}|Dw|^{\kappa}\,dx\right)^{1/\kappa}.
\end{equation}
In particular, it follows that $Dw \in L^{2q-p}_{\loc}(B_{r}) \subset L^{q}_{\loc}(B_{r})$. 
\end{lemma}
\begin{proof}
It suffices to prove \eqref{higher.est} for $\tilde{q} = np/(n-2\beta)$ for $\beta \in (n(q/p-1),\alpha)$. Observe that we can obtain the following estimate as a variant of \cite[(5.45)]{CM15b}:
\begin{equation}\label{CMest}
\left(\mean{B_{r/2}}|Du|^{np/(n-2\beta)}\,dx\right)^{(n-2\beta)/np} \le c\mathcal{B}\left(\mean{B_{3r/4}}|Du|^{p}\,dx\right)^{1/p},
\end{equation}
where $c\equiv c(n,p,q,\nu,L,\beta)$ and
\begin{equation*}
\begin{aligned}
\mathcal{B} & \coloneqq \left[ 1+ \left(\|a\|_{L^{\infty}(B_{3r/4})}^{2} + r^{2\alpha}[a]_{\alpha}^{2}\right)^{b_{1}}\left(\mean{B_{3r/4}}|Du|^{p}\,dx\right)^{b_{2}-1}\right]^{1/p} \\
\overset{\eqref{Lp.Lkappa}}& {\le} c\left[ 1+ \left(\|a\|_{L^{\infty}(B_{3r/4})}^{2}+r^{2\alpha}[a]_{\alpha}^{2}\right)^{b_{1}}\left(\mean{B_{r}}|Du|^{\kappa}\,dx\right)^{(b_{2}-1)p/\kappa} \right]^{1/p} \\
\overset{\eqref{p.phase.4r}}& {\le} c\left[ 1+ (\bar{L}^{2}+1)^{b_{1}}[a]_{\alpha}^{2b_{1}}\frac{r^{2\alpha b_{1}}}{r^{n(b_{2}-1)p/\kappa}}\left(\int_{B_{r}}|Du|^{\kappa}\,dx\right)^{(b_{2}-1)p/\kappa} \right]^{1/p}
\end{aligned}
\end{equation*}
with
\begin{equation*}
b_{1} \coloneqq \frac{\beta p}{\beta p - n(q-p)} \ge 1, \quad b_{2} \coloneqq \frac{\beta(2q-p)-n(q-p)}{\beta p -n(q-p)} \ge 1.
\end{equation*}
Now, using the fact that
\begin{equation*}
2\alpha b_{1} - \frac{n(b_{2}-1)p}{\kappa} = \frac{2\beta pn}{\beta p-n(q-p)}\left(\frac{\alpha}{n} - \frac{q-p}{\kappa}\right) > 0 \quad \text{by } \eqref{kappa.range}
\end{equation*}
and applying \eqref{Lp.Lkappa} once again to the right-hand side of \eqref{CMest}, 
we conclude with \eqref{higher.est}.
\end{proof}

\subsection{Reverse H\"older estimates in the $(p,q)$-phase}

We now turn our attention to the situation when condition \eqref{p.phase.4r} is violated; that is, we consider the case
\begin{equation}\label{pq.phase.4r}
\sup_{x \in B_{r}} a(x) > \bar{L} [a]_{\alpha} r^{\alpha}.
\end{equation}
This together with \eqref{a.holder} implies 
\begin{equation}\label{a.away}
\sup_{x \in B_{r}} a(x) \leq 2 \inf_{x \in B_{r}} a(x).
\end{equation}
Observe from \eqref{a.away} that the vector field $A(\cdot)$, initially subject to \eqref{growth}, actually satisfies
\begin{equation}\label{pq.growth}
\left\{
\begin{aligned}
|A(x,z)| + |\partial A(x,z)||z| &\le L(|z|^{p-1} + \|a\|_{L^{\infty}(B_{r})}|z|^{q-1}),\\
\frac{\nu}{2} (|z|^{p-2} + \|a\|_{L^{\infty}(B_{r})}|z|^{q-2})|\zeta|^{2} &\le \partial A(x,z) \zeta \cdot \zeta,
\end{aligned}
\right.
\end{equation}
for any $x \in B_r$ and $z , \zeta \in \mathbb{R}^{n}$. 
Although the vector field $A(\cdot)$ depends on the spatial variable $x$, under the assumption \eqref{pq.phase.4r}, the growth and ellipticity conditions \eqref{pq.growth} allow us to employ standard regularity estimates for uniformly elliptic equations with Orlicz growth. 
In particular, a Caccioppoli type estimate is available in this setting, and combining this with the Sobolev--Poincar\'e type inequality for the corresponding Orlicz--Sobolev space yields the following reverse H\"older estimate. 
Since the proof follows standard arguments established in \cite[Theorem 9]{DieEtt2008}, we state the following lemma without proof. 

\begin{lemma}\label{lem.pq.phase}
Let $w$ be a weak solution to \eqref{homogeneous} under assumptions \eqref{growth}, \eqref{p.bound}, \eqref{a.holder}, and \eqref{rate.sola}.
Assume that \eqref{pq.phase.4r} holds for some $\bar{L} \in [8, \infty)$.
Then for any $\varepsilon \in (0,1)$ and any $x_0 \in B_{r/2}$, there is $c \equiv c(n,p,q,\nu,L, \varepsilon)$ satisfying
\[ 
\mean{B_{r/2}} H(x_0, |Dw|) \dx \leq c \left( \mean{B_{r}} [H(x_0, |Dw|)]^{\varepsilon} \dx \right)^{1/\varepsilon}.
\]
\end{lemma}

\begin{remark}
In this remark, we combine the reverse H\"older inequalities obtained so far in this section.
To this aim, we fix any point $x_0 \in B_{r/2}$ and then consider the two functions $\Phi_{1}, \Phi_{2} : [0,\infty) \to \mathbb{R}$ defined by
\[ \Phi_{1}(t) \coloneqq \int_{0}^{t} \dfrac{H(x_0, s^{1/p})}{s} \, ds \qquad \text{and} \qquad \Phi_{2}(t) \coloneqq \int_{0}^{t} \dfrac{H(x_0, s)^{1/q}}{s} \, ds \]
for $t \in [0,\infty)$. 
It is straightforward to see that, for all $ t \in [0,\infty)$,
\[ \Phi_{1}'(t), \Phi_{2}'(t)>0, \quad\Phi_{1}''(t) \geq 0, \quad \text{and} \quad \Phi_{2}''(t) \leq 0. \]
In other words, $\Phi_{1}$ is convex while $\Phi_{2}$ is concave.
Consequently, for any weak solution $w$ to \eqref{homogeneous} under the assumptions of Lemma~\ref{lem.pq.phase}, we have
\begin{align*}
\mean{B_{r/2}} |Dw|^{p} \dx
& \leq c \Phi_{1}^{-1} \left( \mean{B_{r/2}} H(x_0, |Dw|) \dx \right) \\
& \leq c \Phi_{1}^{-1} \left( \left[ \mean{B_{r}} H(x_0, |Dw|)^{1/q} \dx \right]^{q} \right) \\
& \leq c \Phi_{1}^{-1} \left( \left[ \Phi_{2} \left( \mean{B_{r}} |Dw| \dx \right) \right]^{q} \right) \\
& \leq c \left ( \mean{B_{r}} |Dw| \dx \right)^{p}.
\end{align*}

Combining Lemmas~\ref{lem:higher} and \ref{lem.pq.phase}, we deduce that, regardless of the phases, any weak solution $w$ to \eqref{homogeneous} satisfies the following estimate:
\begin{equation*}
\left( \mean{ B_{r/2}} |Dw|^{p} \dx \right)^{1/p} \leq c \left( \mean{B_{r}} |Dw|^{\kappa} \dx \right)^{1/\kappa}, 
\end{equation*}
where $c$ has the same dependence as in Lemma~\ref{lem:higher}. 
Accordingly, the exponent $\kappa$ on the right-hand side of \eqref{higher.est} can then be replaced by any smaller exponent by using standard covering and interpolation-iteration arguments (see for instance \cite[Remark~6.12]{Giu}). 
More precisely, under the assumptions of Lemma \ref{lem:higher}, for any $\tilde{q} < np/(n-2\alpha)$, $\eta \in (0,1)$ and $\varepsilon \in (0,\kappa]$, there exists a constant $c \equiv c(n,p,q,\nu,L, \|a\|_{L^{\infty}}, \alpha, [a]_{\alpha}, \bar{L}, \|Dw\|_{L^{\kappa}(B_{r})},\tilde{q},\eta, \varepsilon)$ such that
\begin{equation}\label{higher.combined} 
\left(\mean{B_{\eta r}} |Dw|^{\tilde{q}} \, dx\right)^{1/\tilde{q}} \le c\left(\mean{B_{r}}|Dw|^{\varepsilon}\,dx\right)^{1/\varepsilon}. 
\end{equation}
\end{remark}

\subsection{Lipschitz regularity for frozen problems}
We end this section with regularity estimates for homogeneous equations with a specific class of Orlicz growth. Consider a $C^{1}$-vector field $A_{0}:\mathbb{R}^{n}\rightarrow\mathbb{R}^{n}$ satisfying
\begin{equation}\label{growth.fixed}
\left\{
\begin{aligned}
|A_{0}(z)| + |\partial A_{0}(z)||z| &\le L(|z|^{p-1}+a_{0}|z|^{q-1}), \\
\nu(|z|^{p-2}+a_{0}|z|^{q-2})|\xi|^{2} &\le \partial A_{0}(z)\zeta\cdot\zeta
\end{aligned}
\right.
\end{equation}
for any $z,\zeta \in \mathbb{R}^{n}$, where $a_{0} \ge 0$ is a fixed constant. 
The following local Lipschitz estimate can be found in  \cite{Ba15,Lie91}.

\begin{lemma}\label{lem.Lip}
Let $v$ be a weak solution to 
\begin{equation*}
-\mathrm{div}\,A_{0}(Dv)=0 \quad \textrm{in}\;\;\Omega
\end{equation*} 
under assumptions \eqref{growth.fixed} and \eqref{p.bound}. Then $Dv \in L^{\infty}_{\loc}(\Omega)$. Moreover, there exists a constant $c \equiv c(n,p,q,\nu,L)$, which is in particular independent of $a_{0}$,  such that
\begin{equation*}
\sup_{B_{r}}|Dv| \le c\mean{B_{2r}}|Dv|\,dx
\end{equation*}
holds whenever $B_{2r} \Subset \Omega$ is a ball.
\end{lemma}

\section{Proof of the main theorem}\label{sec4}

In this section, we prove Theorem~\ref{main.thm}. 
We first recall the definition of $\mathbf{M}_1$ given in \eqref{def.M1}. Observe that there exists a constant  $c_{0} \equiv c_{0}(n) \ge 1$ such that
\begin{equation}\label{ineq.M1}
\frac{|\mu|(\overline{B_{r}(x_{0})})}{r^{n-1}}  \le c_{0}\mean{B_{r}(x_{0})}\left[\frac{|\mu|(B_{2r}(x))}{(2r)^{n-1}}\right]\,dx 
 \le c_{0}\mean{B_{r}(x_{0})}\mathbf{M}_{1}(\mu)\,dx
\end{equation}
holds whenever $B_{r}(x_{0}) \subset \mathbb{R}^{n}$ is a ball.

We also recall a well-known iteration lemma, see for instance \cite[Lemma 6.1]{Giu}.

\begin{lemma}\label{tech.lemma}
Let $Z:[R,2R] \rightarrow [0,\infty)$ be a bounded function that satisfies
\begin{equation*}
Z(r_{1}) \le \varepsilon_{0}Z(r_{2}) + \frac{C_{1}}{(r_{2}-r_{1})^{\ell}} + C_{2}
\end{equation*}
for any $r_{1},r_{2}$ with $R \le r_{1} < r_{2} \le 2R$, where $\varepsilon_{0} \in (0,1)$ and $C_{1},C_{2},\ell > 0$ are given constants. Then there exists a constant $c \equiv c(\varepsilon_{0},\ell)$ such that
\begin{equation*}
Z(R) \le c\left[\frac{C_{1}}{R^{\ell}} + C_{2}\right].
\end{equation*}
\end{lemma}

We are now ready to prove our main theorem.

\subsection{Proof of Theorem~\ref{main.thm}}
The proof is divided into eleven steps.

\textit{Step 1: Exit time and covering arguments. } 
We take a ball $B_{2R}  \Subset \Omega$ with $2R \le R_0$ as in the statement of the theorem. We initially assume that $R_0 \le 1$ satisfies
\begin{equation}\label{R0.ini}
R_{0}^{\alpha-(n-1)\frac{q-p}{p-1}}[|\mu|(\Omega)]^{\frac{q-p}{p-1}} \le 1 \qquad \text{and} \qquad R_{0}^{n-\frac{\kappa(n-1)}{p-1}}[|\mu|(\Omega)]^{\frac{\kappa}{p-1}} \le 1 ,
\end{equation} 
which is possible due to \eqref{rate.sola} and \eqref{def.kappa}. 
The value of $R_0$ will be determined later in the proof. We select two radii $r_{1},r_{2}$ such that $R \le r_{1} < r_{2} \le 2R$ and consider the upper level sets
\begin{equation*}
E(s,\lambda) \coloneqq \{x \in B_{s}: h(x,|Du(x)|)>\lambda\} \qquad \text{for any} \;\; R\le s \le 2R \;\; \text{and} \;\; \lambda>0.
\end{equation*} 
With $M \ge 1$ being a free parameter, whose value will be eventually determined in the end of the proof, we define
\begin{equation*}
\Psi(B_{\rho}(x_{0})) \coloneqq \mean{B_{\rho}(x_{0})}\left[h(x,|Du|)+M\mathbf{M}_{1}(\mu)\right]\,dx
\end{equation*}
for any ball $B_{\rho}(x_0) \subset B_{2R}$. We then observe that
\begin{equation}\label{psi.limit}
\lim_{\rho\searrow0}\Psi(B_{\rho}(x_{0}))> \lambda \qquad \text{for a.e.} \;\; x_{0} \in E(s,\lambda) \;\; \text{and} \;\; R \le s \le 2R.
\end{equation}
Also, if $x_{0} \in B_{r_{1}}$, then for any $\rho \in [(r_{2}-r_{1})/20, r_{2}-r_{1}]$ we have
\begin{equation}\label{def.lambda0}
\Psi(B_{\rho}(x_{0})) \le \frac{20^{n}r_{2}^{n}}{(r_{2}-r_{1})^{n}}\mean{B_{r_{2}}}\left[h(x,|Du|) + M\mathbf{M}_{1}(\mu)\right]\,dx \eqqcolon \lambda_{0}.
\end{equation}

From now on, we consider 
\begin{equation}\label{lambda.range}
\lambda > \lambda_0. 
\end{equation}
From \eqref{psi.limit} and \eqref{def.lambda0}, it follows that there exists an exit time radius $\rho_{x_{0}} \in (0, (r_{2}-r_{1})/20)$ such that
\begin{equation*}
\Psi(B_{\rho_{x_{0}}}(x_{0})) = \lambda \qquad \text{and} \qquad \Psi(B_{\rho}(x_{0})) < \lambda \quad \text{for every } \rho \in (\rho_{x_{0}},r_{2}-r_{1}].
\end{equation*} 
Since this holds for a.e. $x_0 \in E(r_1,\lambda)$, we see that
the family $\{B_{\rho_{x_{0}}}(x_{0})\}$ covers $E(r_1,\lambda)$ up to a negligible set. Thus, an application of Vitali's covering lemma implies that there exists a countable family $\{B_{\rho_{x_i}}(x_i)\} \equiv \{\tilde{B}_{i}\}$ of mutually disjoint balls satisfying
\begin{equation}\label{covering}
E(r_1 ,\lambda) \subset \bigcup_{i}5\tilde{B}_{i} \cup \text{negligible set} 
\end{equation}
and, whenever $i \in \mathbb{N}$,
\begin{equation}\label{exit.radius}
\Psi(B_{\rho_{x_i}}(x_i)) = \lambda \qquad \text{and} \qquad \Psi(B_{\rho}(x_i))<\lambda \quad \text{for every } \rho \in (\rho_{x_i}, r_{2}-r_{1}].
\end{equation}
In the following, we simply denote
\begin{equation}\label{def.Bi} 
B_{i} \coloneqq 5B_{\rho_{x_i}}(x_i) = 5\tilde{B}_{i} \qquad \text{and} \qquad \rho_i \coloneqq 5\rho_{x_i}. 
\end{equation}
Observe that
\begin{equation*}
\rho_{i} = 5\rho_{x_i} \le \frac{r_{2}-r_{1}}{4} \le R_0 \le 1
\end{equation*}
and that, by construction, 
\begin{equation*}
20\tilde{B}_{i} = 4B_{i} \subset B_{r_2}.
\end{equation*}
We also note that \eqref{exit.radius} implies
\begin{equation}\label{exit2}
\left\{
\begin{aligned}
 \Psi(\tilde{B}_{i}) & = \mean{\tilde{B}_{i}}\left[ h(x,|Du|) + M\mathbf{M}_{1}(\mu)\right]\,dx = \lambda, \\
 \Psi(4B_{i}) & = \mean{4B_{i}}\left[ h(x,|Du|) + M\mathbf{M}_{1}(\mu)\right]\,dx \le \lambda.
\end{aligned}
\right.
\end{equation}

\textit{Step 2: Approximation and a first comparison function. }
With $B_i$ being any fixed ball considered in the above, we choose a sequence of functions $\{\mu_j\} \subset L^{\infty}(\Omega)$ and a sequence of corresponding weak solutions $\{u_{j}\}$ to \eqref{approx.prob}, which satisfy the convergence properties described in Proposition~\ref{prop.existence}. In particular, we can choose $j \in \mathbb{N}$ large enough to satisfy
\begin{equation}\label{approx.Du.2}
\|Du_j \|_{L^{\kappa}(4B_{i})}  \le 2 \|Du \|_{L^{\kappa}(4B_i)},
\end{equation}
\begin{equation}\label{approx.Du}
\mean{4B_i}h(x,|Du_{j}-Du|)\,dx \le  \frac{\lambda_0}{M}
\end{equation}

and
\begin{equation}\label{approx.mu}
|\mu_j|(4B_i) \le 2|\mu|(\overline{4B_i}).
\end{equation}
Accordingly, we consider the weak solution $w_{i,j} \in u_{j} + W^{1,H}_{0}(4B_i)$ to the Dirichlet problem
\begin{equation*}
\left\{
\begin{aligned}
-\mathrm{div}\, A(x,Dw_{i,j})& = 0&\text{in }& 4B_{i}, \\
w_{i,j} & = u_{j}& \text{on }& \partial (4B_i).
\end{aligned}
\right.
\end{equation*}
In this step, we prove the following comparison estimate
\begin{equation}\label{comp1}
\mean{4B_i}h(x,|Du_{j}-Dw_{i,j}|)\,dx \le c\left[\frac{|\mu|(\overline{4B_i})}{(4\rho_i)^{n-1}}\right],
\end{equation}
where $c\equiv c(n,p,q,\nu,L, \|a\|_{L^{\infty}}, \alpha, [a]_{\alpha})$. For this we consider the following two cases:
\begin{equation}\label{alt.comp1}
\inf_{x\in 4B_{i}}a(x) > 8[a]_{\alpha}\rho_{i}^{\alpha} \qquad \text{and} \qquad \inf_{x\in 4B_{i}}a(x) \le 8[a]_{\alpha}\rho_{i}^{\alpha}.
\end{equation}

In the case $\eqref{alt.comp1}_{1}$, we see that 
\[ \inf_{x\in 4B_i}a(x) \le a(x) \le \inf_{x\in 4B_i}a(x) + [a]_{\alpha}(8\rho_i)^{\alpha} \le 2\inf_{x\in 4B_i}a(x) \] 
holds for any $x \in 4B_i$. Thus, recalling \eqref{pq.growth}, we can obtain estimate \eqref{comp1} in a way similar to those available for equations with Orlicz growth, see for instance \cite[Lemma 5.1]{Ba15}. 

In the case $\eqref{alt.comp1}_{2}$, we observe
\begin{equation}\label{alt.comp1.a}
a(x) \le \inf_{x\in 4B_i}a(x) + [a]_{\alpha}(8\rho_i)^{\alpha} \le 16[a]_{\alpha}\rho_i^{\alpha} \quad \text{for any }x\in 4B_i. 
\end{equation}
Here we revisit the argument used in the proof of Lemma \ref{global.est}. With $d>0$ and $\xi >1$ being any numbers, we use $d^{1-\xi} - (d+(u_{j}-w_{i,j})_{\pm})^{1-\xi}$ as a test function in
\[ -\mathrm{div}\,(A(x,Du_j)-A(x,Dw_{i,j})) = \mu_j \quad \text{in}\; 4B_i . \]
Then,  using \eqref{mono} and then 
replacing $u$, $\mu$ and $\Omega$ in the proof of Lemma \ref{global.est} with $u_j - w_{i,j}$, $\mu_j$ and $4B_{i}$, respectively,
we get
\[ \int_{4B_i}\frac{(|Du_j|+|Dw_{i,j}|)^{p-2}|Du_j - Dw_{i,j}|^2}{(d+|u_j - w_{i,j}|)^{\xi}}\,dx \le c\frac{d^{1-\xi}}{\xi-1}|\mu_j|(4B_i) \]
for some $c\equiv c(n,p,q,\nu,L) \ge 1$, whenever $d>0$ and $\xi >1$. 
Once we have the above estimate, we can follow the proof of  \cite[Lemma 2]{KM14BMS} in order to obtain
\begin{equation}\label{comp.p}
\mean{4B_{i}}|Du_{j}-Dw_{i,j}|^{s}\,dx \le c\left[\frac{|\mu_j|(4B_{i})}{(4\rho_i)^{n-1}}\right]^{s/(p-1)}
\end{equation}
for any $s$ satisfying \eqref{exp.range}, where $c\equiv c(n,p,q,\nu,L,s)$. In particular, we have
\begin{equation*}
\mean{4B_{i}}|Du_{j}-Dw_{i,j}|^{p-1}\,dx \le c\left[\frac{|\mu_j|(4B_{i})}{(4\rho_{i})^{n-1}}\right].
\end{equation*}
Now, recalling \eqref{q-1.admissible}, we can take $s = q-1$ in \eqref{comp.p}, which along with \eqref{alt.comp1.a} implies
\begin{equation*} 
\begin{aligned}
\mean{4B_{i}}a(x)|Du_{j}-Dw_{i,j}|^{q-1}\,dx & \le c[a]_{\alpha}\rho_{i}^{\alpha}\mean{4B_{i}}|Du_{j}-Dw_{i,j}|^{q-1}\,dx \\
& \le c[a]_{\alpha}\rho_{i}^{\alpha}\left[\frac{|\mu_j|(4B_{i})}{(4\rho_i)^{n-1}}\right]^{\frac{q-1}{p-1}} \\
& = c[a]_{\alpha}\rho_{i}^{\alpha}\left[\frac{|\mu_j|(4B_{i})}{(4\rho_{i})^{n-1}}\right]^{\frac{q-p}{p-1}}\left[\frac{|\mu_j|(4B_{i})}{(4\rho_{i})^{n-1}}\right].
 \end{aligned}
\end{equation*}
Then, due to \eqref{rate.sola}, \eqref{approx.mu} and the fact that $r\le 1$, we get
\begin{equation*}
\begin{aligned}
\rho_{i}^{\alpha}\left[\frac{|\mu_j|(4B_{i})}{(4\rho_{i})^{n-1}}\right]^{\frac{q-p}{p-1}}
& \le c\rho_{i}^{\alpha-(n-1)\frac{q-p}{p-1}}[|\mu|(\Omega)]^{\frac{q-p}{p-1}} \overset{\eqref{R0.ini}}{\le} c(n,p,q,\alpha). 
\end{aligned}
\end{equation*}
Combining the above three displays, and then using \eqref{approx.mu}, we conclude with  \eqref{comp1}. Note that \eqref{comp1} holds in both case $\eqref{alt.comp1}_1$ and $\eqref{alt.comp1}_2$.

\textit{Step 3: A second comparison function. } 
With the same ball $4B_{i}$ as in the previous step, we choose a point $x_{i,m} \in \overline{2B_{i}}$ satisfying
\[ a(x_{i,m}) =  \sup_{x \in 2B_{i}}a(x). \]
We then consider the following homogeneous frozen problem: 
\begin{equation}\label{def.v}
\left\{
\begin{aligned}
-\mathrm{div}\, A(x_{i,m},Dv_{i,j}) &= 0 &\textrm{in }& 2B_{i}, \\
v_{i,j} &= w_{i,j} &\textrm{on }& \partial (2B_{i}).
\end{aligned}
\right.
\end{equation}
Since $w_{i,j} \in W^{1,q}(2B_i)$ by Lemma~\ref{lem:higher}, 
the existence and uniqueness of a weak solution $v_{i,j}$ to \eqref{def.v} follows from standard monotonicity methods in Orlicz spaces. Moreover, $v_{i,j}$ satisfies the energy estimate
\begin{equation} \label{v.energy}
\mean{2B_{i}}(|Dv_{i,j}|^{p}+a(x_{i,m})|Dv_{i,j}|^{q})\,dx \le c\mean{2B_{i}}(|Dw_{i,j}|^{p}+a(x_{i,m})|Dw_{i,j}|^{q})\,dx 
\end{equation}
for some $c\equiv c(n,p,q,\nu,L)$.

As in \cite[(4.23)--(4.26)]{CM16JFA}, we estimate
\begin{equation}\label{comp.I}
\begin{aligned}
& \mean{2B_i}(|Dw_{i,j}-Dv_{i,j}|^{p}+a(x_{i,m})|Dw_{i,j}-Dv_{i,j}|^{q})\,dx  \\
& \le c\left(\osc_{2B_i}a\right)\mean{2B_i}|Dw_{i,j}|^{q-1}|Dv_{i,j}-Dw_{i,j}|\,dx \eqqcolon I
\end{aligned}
\end{equation}
for some $c\equiv c(n,p,q,\nu,L)$; note that we have also used  \eqref{mono2}, since $2 \le p < q$.

\textit{Step 4: Two different phases. } 
In order to estimate $I$, we divide the cases as follows:
\begin{equation}\label{alt.comp2}
\inf_{x\in 2B_i}a(x) > K[a]_{\alpha}\rho_{i}^{\alpha} \qquad \text{and} \qquad \inf_{x \in 2B_i}a(x) \le K[a]_{\alpha}\rho_{i}^{\alpha},
\end{equation}
where $K \ge 4$ is a free parameter whose value will be determined later in the proof.

\textit{Step 5: Estimates in the $(p,q)$-phase. } Here we consider the case $\eqref{alt.comp2}_{1}$. 
In this case, we have 
\[ \osc_{2B_{i}}a \le [a]_{\alpha}(4\rho_i)^{\alpha} \le \frac{4}{K}a(x) \] 
and 
\begin{equation}\label{a.equiv}
a(x_{i,m}) \le a(x) + [a]_{\alpha}(2\rho_i)^{\alpha} \le a(x) + \frac{4}{K}a(x) \le 2a(x)
\end{equation}
 for any $x\in 2B_{i}$; we thus obtain
\begin{equation}\label{pq.phase.I}
\begin{aligned}
I & \le \frac{c}{K}\mean{2B_{i}}a(x)|Dw_{i,j}|^{q-1}|Dv_{i,j}-Dw_{i,j}|\,dx \\
& \le \frac{c}{K}\mean{2B_{i}}a(x)(|Dv_{i,j}|+|Dw_{i,j}|)^{q}\,dx 
\overset{\eqref{v.energy}} {\le} \frac{c}{K}\mean{2B_{i}}(|Dw_{i,j}|^{p}+a(x_{i,m})|Dw_{i,j}|^{q})\,dx 
\end{aligned}
\end{equation}
for some $c\equiv c(n,p,q,\nu,L)$. 
For brevity, we write 
\[ H(x_{i,m},t) \eqqcolon H_{x_{i,m}}(t) \quad \text{and} \quad h(x_{i,m},t) \eqqcolon h_{x_{i,m}}(t) \qquad \text{for} \;\; t \ge 0. \] 
Observe that both $H_{x_{i,m}}$ and $H_{x_{i,m}} \circ h_{x_{i,m}}^{-1}$ are convex. Connecting \eqref{pq.phase.I} to \eqref{comp.I}, we get
\begin{equation*}
\mean{2B_{i}}H_{x_{i,m}}(|Dw_{i,j}-Dv_{i,j}|)\,dx \le \frac{c}{K}\mean{2B_{i}}H_{x_{i,m}}(|Dw_{i,j}|)\,dx.
\end{equation*}
We also note that, due to \eqref{a.equiv}, we can apply Lemma \ref{lem.pq.phase} to $w_{i,j}$, which (after choosing $\varepsilon>0$ sufficiently small) gives
\[ \mean{2B_i}H_{x_{i,m}}(|Dw_{i,j}|)\,dx \le cH_{x_{i,m}}\left(\mean{4B_i}|Dw_{i,j}|\,dx\right)  \]
for some $c\equiv c(n,p,q,\nu,L)$.
Combining the above two displays, and then using Jensen's inequality, we obtain
\begin{equation}\label{conclusion.pq}
\begin{aligned}
\mean{2B_{i}}h_{x_{i,m}}(|Dw_{i,j}-Dv_{i,j}|)\,dx & \le (h_{x_{i,m}}\circ H_{x_{i,m}}^{-1})\left(\mean{2B_{i}}H_{x_{i,m}}(|Dw_{i,j}-Dv_{i,j}|)\,dx\right) \\
& \le c(h_{x_{i,m}}\circ H_{x_{i,m}}^{-1})\left(\frac{1}{K}\mean{2B_{i}}H_{x_{i,m}}(|Dw_{i,j}|)\,dx\right) \\
& \le \frac{c}{K^{(p-1)/q}}\mean{4B_{i}}h_{x_{i,m}}(|Dw_{i,j}|)\,dx \\
\overset{\eqref{a.equiv}}&{\le} \frac{\tilde{c}}{K^{(p-1)/q}}\mean{4B_i}h(x,|Dw_{i,j}|)\,dx
\end{aligned}
\end{equation}
for some $\tilde{c}\equiv \tilde{c}(n,p,q,\nu,L)$.

\textit{Step 6: Estimates in the $p$-phase. } We next assume $\eqref{alt.comp2}_{2}$ and note that
\begin{equation}\label{p.phase.am}
a(x_{i,m}) \le \inf_{x \in 2B_i}a(x) + [a]_{\alpha}(2\rho_i)^{\alpha} \le (4+K)[a]_{\alpha}(2\rho_i)^{\alpha},
\end{equation}
which enables us to use Lemma~\ref{lem:higher}. 
Applying the obvious inequality $\osc_{2B_i}a \le a(x_{i,m})$ and then Young's inequality to the right-hand side of \eqref{comp.I}, we get
\begin{equation*}
\begin{aligned}
& \mean{2B_i}(|Dw_{i,j}-Dv_{i,j}|^{p} + a(x_{i,m})|Dw_{i,j}-Dv_{i,j}|^{q})\,dx \\
 & \le ca(x_{i,m})\mean{2B_{i}}|Dw_{i,j}|^{q-1}|Dw_{i,j}-Dv_{i,j}|\,dx \\
& \le c\mean{2B_{i}}a(x_{i,m})|Dw_{i,j}|^{q}\,dx + \frac{1}{2}\mean{2B_{i}}a(x_{i,m})|Dw_{i,j}-Dv_{i,j}|^{q}\,dx
\end{aligned}
\end{equation*}
and so
\begin{equation}\label{p.phase.start}
\mean{2B_i}(|Dw_{i,j}-Dv_{i,j}|^{p}+a(x_{i,m})|Dw_{i,j}-Dv_{i,j}|^{q})\,dx \le ca(x_{i,m})\mean{2B_{i}}|Dw_{i,j}|^{q}\,dx
\end{equation}
for a constant $c\equiv c(n,p,q,\nu,L)$. 
We thus have
\begin{equation*}
\begin{aligned}
\mean{2B_{i}}|Dw_{i,j}-Dv_{i,j}|^{p-1}\,dx &\le \left(\mean{2B_{i}}|Dw_{i,j}-Dv_{i,j}|^{p}\,dx\right)^{(p-1)/p} \\
\overset{\eqref{p.phase.start}}&{\le} c\left(a(x_{i,m})\mean{2B_{i}}|Dw_{i,j}|^{q}\,dx\right)^{(p-1)/p} \\
\overset{\eqref{p.phase.am},\eqref{higher.combined}}&{\le} c\left[\rho_{i}^{\alpha}\left(\mean{3B_i}|Dw_{i,j}|^{\kappa}\,dx\right)^{q/\kappa}\right]^{(p-1)/p} \\
& = c\left[\rho_{i}^{\alpha}\left(\mean{3B_i}|Dw_{i,j}|^{\kappa}\,dx\right)^{(q-p)/\kappa + p/\kappa}\right]^{(p-1)/p} \\
& \le c\left[\rho_{i}^{\alpha-n(q-p)/\kappa}\|Dw_{i,j}\|_{L^{\kappa}(3B_i)}^{q-p}\left(\mean{3B_i}|Dw_{i,j}|^{\kappa}\,dx\right)^{p/\kappa}\right]^{(p-1)/p} \\
\overset{\eqref{higher.combined}}&{\le} c\left[\rho_{i}^{\alpha-n(q-p)/\kappa}\|Dw_{i,j}\|_{L^{\kappa}(4B_i)}^{q-p}\right]^{(p-1)/p}\mean{4B_i}|Dw_{i,j}|^{p-1}\,dx
\end{aligned}
\end{equation*}
and
\begin{equation*}
\begin{aligned}
\mean{2B_{i}}a(x_{i,m})|Dw_{i,j}-Dv_{i,j}|^{q-1}\,dx &\le [a(x_{i,m})]^{1/q}\left(\mean{2B_i}a(x_{i,m})|Dw_{i,j}-Dv_{i,j}|^{q}\,dx\right)^{(q-1)/q} \\
\overset{\eqref{p.phase.start}}& {\le} ca(x_{i,m})\left(\mean{2B_{i}}|Dw_{i,j}|^{q}\,dx\right)^{(q-1)/q} \\
\overset{\eqref{p.phase.am},\eqref{higher.combined}}& {\le} c\rho_{i}^{\alpha}\left(\mean{3B_{i}}|Dw_{i,j}|^{\kappa}\,dx\right)^{(q-1)/\kappa} \\
& = c\rho_{i}^{\alpha}\left(\mean{3B_{i}}|Dw_{i,j}|^{\kappa}\,dx\right)^{(q-p)/\kappa + (p-1)/\kappa} \\
& \le c\rho_{i}^{\alpha-n(q-p)/\kappa}\|Dw_{i,j}\|_{L^{\kappa}(3B_i)}^{q-p}\left(\mean{3B_{i}}|Dw_{i,j}|^{\kappa}\,dx\right)^{(p-1)/\kappa} \\
\overset{\eqref{higher.combined}}&{\le} c\rho_{i}^{\alpha-n(q-p)/\kappa}\|Dw_{i,j}\|_{L^{\kappa}(4B_i)}^{q-p}\mean{4B_i}|Dw_{i,j}|^{p-1}\,dx,
\end{aligned}
\end{equation*}
where $c\equiv c(n,p,q,\nu,L, \|a\|_{L^{\infty}}, \alpha, [a]_{\alpha},\|Dw_{i,j}\|_{L^{\kappa}(4B_i)},K)$. Combining the above two displays, and using the fact that \eqref{R0.ini}, \eqref{approx.Du.2},  \eqref{approx.mu} and \eqref{comp.p} imply
\begin{equation}\label{Dw.data}
\begin{aligned}
\|Dw_{i,j}\|_{L^{\kappa}(4B_i)} & \le \|Dw_{i,j}-Du_j\|_{L^{\kappa}(4B_i)} + \|Du_j\|_{L^{\kappa}(4B_i)} \\
& \le c\rho_{i}^{n-\kappa(n-1)/(p-1)}[|\mu_j|(4B_i)]^{\kappa/(p-1)} + \|Du_j\|_{L^{\kappa}(4B_i)} \\
& \le c(\data),
\end{aligned}
\end{equation} 
we get
\begin{equation}\label{conclusion.p}
\mean{2B_i}h(x_{i,m} ,|Dw_{i,j}-Dv_{i,j}|)\,dx \le c_{K}\rho_{i}^{\sigma}\mean{4B_{i}}h(x,|Dw_{i,j}|)\,dx
\end{equation}
for some $c_{K}\equiv c_{K}(\data,K)$, where
\begin{equation}\label{def.sigma} 
\sigma \coloneqq \frac{p-1}{p}\left(\alpha-\frac{n(q-p)}{\kappa}\right) > 0 \quad \text{by }\, \eqref{kappa.range}. 
\end{equation}

\textit{Step 7: Matching the two phases and comparison estimates. }
Combining \eqref{conclusion.pq} and \eqref{conclusion.p}, we have that
\begin{equation*}
\begin{aligned}
\mean{2B_i}h(x_{i,m},|Dw_{i,j}-Dv_{i,j}|)\,dx 
\le \left[\frac{\tilde{c}}{K^{(p-1)/q}} + c_{K}\rho_{i}^{\sigma}\right]\mean{4B_i}h(x,|Dw_{i,j}|)\,dx
\end{aligned}
\end{equation*}
holds with $\tilde{c}\equiv \tilde{c}(n,p,q,\nu,L)$ and $c_{K} \equiv c_{K}(\data,K)$. 
Here, if $R_0$ is so small that
\begin{equation}\label{R0.ini2}
c_{K}R_{0}^{\sigma} \le 1,
\end{equation}
then
\begin{equation*}
\begin{aligned}
& \mean{2B_i}h(x_{i,m},|Dw_{i,j}-Dv_{i,j}|)\,dx \\
& \le 2^{q-2}\left[\frac{\tilde{c}}{K^{(p-1)/q}} + c_{K}\rho_{i}^{\sigma}\right]\left(\mean{4B_i}h(x,|Dw_{i,j}-Du_j|)\,dx + \mean{4B_i}h(x,|Du_j|)\,dx\right) \\
\overset{\eqref{comp1}}&{\le} c\left[\frac{|\mu|(\overline{4B_i})}{(4\rho_i)^{n-1}}\right] + 2^{q-2}\left[\frac{\tilde{c}}{K^{(p-1)/q}} + c_{K}\rho_{i}^{\sigma}\right]\mean{4B_i}h(x,|Du_{j}|)\,dx 
\end{aligned}
\end{equation*}
holds for some $c\equiv c(\data)$, since $K \ge 4$.
Combining the above display with \eqref{comp1}, we arrive at
\begin{equation}\label{comp.comb}
\begin{aligned}
& \mean{2B_i}h(x,|Du_{j}-Dv_{i,j}|)\,dx \\
& \le 2^{q-2}\mean{2B_i}h(x,|Du_j - Dw_{i,j}|)\,dx + 2^{q-2}\mean{2B_i}h(x_{i,m},|Dw_{i,j}-Dv_{i,j}|)\,dx \\
& \le  c_{*}\left[\frac{|\mu|(\overline{4B_i})}{(4\rho_i)^{n-1}}\right] + 2^{2(q-2)}\left[\frac{\tilde{c}}{K^{(p-1)/q}} + c_{K}R_{0}^{\sigma}\right]\mean{4B_i}h(x,|Du_j|)\,dx,
\end{aligned}
\end{equation}
where $c_{*}\equiv c_{*}(\data)$, $\tilde{c}\equiv \tilde{c}(n,p,q,\nu,L)$ and $ c_{K} \equiv c_{K}(\data,K)$. 
With these constants and $c_{0} \equiv c_{0}(n)$ given in \eqref{ineq.M1},  we now introduce the notation
\begin{equation}\label{def.S}
 S(R_0, K,M) \coloneqq 2^{4(q-2)}\left[\frac{\tilde{c}}{K^{(p-1)/q}} + c_{K}R_{0}^{\sigma}\right] + \frac{c_{0}c_{*}+1}{M}.
\end{equation}
Then, applying \eqref{ineq.M1}, \eqref{lambda.range}, \eqref{exit2}, and \eqref{approx.Du} to \eqref{comp.comb}, we have that
\begin{equation}\label{final.comp}
\mean{2B_i}h(x,|Du_{j}-Dv_{i,j}|)\,dx \le S(R_0,K,M)\lambda
\end{equation}
holds for any ball $B_i$ from the covering given in \eqref{def.Bi}. 

\textit{Step 8: The two phases at a different threshold. }
Here we prove that
\begin{equation}\label{h.Dv.lambda}
\mean{2B_i}h(x_{i,m},|Dv_{i,j}|)\,dx  \le c\lambda
\end{equation}
holds for a constant $c\equiv c(\data)$.  To this aim, we start by estimating
\begin{equation*}
\mean{2B_i}h(x_{i,m},|Dv_{i,j}|)\,dx \le c\mean{2B_i}h(x_{i,m},|Dw_{i,j}-Dv_{i,j}|)\,dx + c\mean{2B_i}h(x_{i,m},|Dw_{i,j}|)\,dx.
\end{equation*}
Then, we again consider the following two alternatives:
\begin{equation}\label{alt.2Bi}
\inf_{x \in 2B_i}a(x) > 10[a]_{\alpha}\rho_{i}^{\alpha} \qquad \text{and} \qquad \inf_{x\in 2B_i}a(x) \le 10[a]_{\alpha}\rho_{i}^{\alpha},
\end{equation}
which are nothing but $\eqref{alt.comp2}_{1}$ and $\eqref{alt.comp2}_{2}$, respectively, with $K=10$. 

If $\eqref{alt.2Bi}_{1}$ holds, then we have
\begin{equation*}
\mean{2B_i}h(x_{i,m},|Dw_{i,j}-Dv_{i,j}|)\,dx \overset{\eqref{conclusion.pq}}{\le} c\mean{4B_i}h(x,|Dw_{i,j}|)\,dx
\end{equation*}
and
\begin{equation}\label{h.Dw.lambda}
\begin{aligned}
& \mean{2B_i}h(x_{i,m},|Dw_{i,j}|)\,dx \overset{\eqref{a.equiv}}{\le} c\mean{4B_i}h(x,|Dw_{i,j}|)\,dx \\
& \le c\mean{4B_i}h(x,|Du_j - Dw_{i,j}|)\,dx + c\mean{4B_i}h(x,|Du_j|)\,dx \\
\overset{\eqref{comp1}}& {\le} c\left[\frac{|\mu|(\overline{4B_i})}{(4\rho_i)^{n-1}}\right] + c\mean{4B_i}h(x,|Du-Du_j|)\,dx + c\mean{4B_i}h(x,|Du|)\,dx \\
& \le c\lambda,
\end{aligned}
\end{equation}
where we have also used \eqref{ineq.M1}, \eqref{approx.Du}, \eqref{def.lambda0} and \eqref{lambda.range} for the last inequality. 
The above two displays lead to \eqref{h.Dv.lambda}.

If $\eqref{alt.2Bi}_{2}$ holds, then we estimate
\begin{equation*}
\begin{aligned}
a(x_{i,m})\mean{2B_i}|Dw_{i,j}|^{q-1}\,dx \overset{\eqref{p.phase.am},\eqref{higher.est}}&{\le} c\rho_{i}^{\alpha}\left(\mean{4B_i}|Dw_{i,j}|^{\kappa}\,dx\right)^{(q-1)/\kappa} \\
& = c\rho_{i}^{\alpha}\left(\mean{4B_i}|Dw_{i,j}|^{\kappa}\,dx\right)^{(q-p)/\kappa + (p-1)/\kappa} \\
& = c\rho_{i}^{\alpha-n(q-p)/\kappa}\|Dw_{i,j}\|^{q-p}_{L^{\kappa}(4B_i)}\left(\mean{4B_i}|Dw_{i,j}|^{\kappa}\,dx\right)^{(p-1)/\kappa} \\
\overset{\eqref{kappa.range},\eqref{Dw.data}}&{\le} c\mean{4B_i}|Dw_{i,j}|^{p-1}\,dx,
\end{aligned}
\end{equation*}
which gives \eqref{h.Dw.lambda} in this case as well.
Using this and \eqref{conclusion.p}, we again obtain \eqref{h.Dv.lambda}.

\textit{Step 9: A priori estimates for $Dv_{i,j}$. } 
Note that estimate \eqref{h.Dv.lambda} is independent of the cases $\eqref{alt.2Bi}_1$ and $\eqref{alt.2Bi}_2$ coming into the play.  Accordingly, here we apply Lemma \ref{lem.Lip} to $v_{i,j}$, with the choice $a_{0} \equiv a(x_{i,m})$, which yields
\begin{equation*}
 \sup_{x \in B_i}h(x_{i,m},|Dv_{i,j}|) \le c\mean{2B_i}h(x_{i,m},|Dv_{i,j}|)\,dx \le c_{l}\lambda
 \end{equation*}
for some $c_{l}\equiv c_{l}(\data)$. 
This together with the definition of $x_{i,m}$ implies
\begin{equation}\label{lip.lambda}
\sup_{x \in B_i}h(x,|Dv_{i,j}|) \le c_{l}\lambda.
\end{equation}

\textit{Step 10: Estimates on upper level sets. } 
By using \eqref{lip.lambda}, we have
\begin{equation*}
\begin{aligned}
& 3^{q-2}c_{l}\lambda|B_{i} \cap \{h(x,|Du|) > 2\cdot3^{q-2}c_{l}\lambda\}| + \frac{1}{2}\int_{B_{i}\cap\{h(x,|Du|)>2\cdot3^{q-2}c_{l}\lambda\}}h(x,|Du|)\,dx \\
& \le \int_{B_{i}\cap\{h(x,|Du|)>2\cdot3^{q-2}c_{l}\lambda\}}h(x,|Du|)\,dx \\
& \le 3^{q-2}\int_{B_i}h(x,|Du-Du_{j}|)\,dx + 3^{q-2}\int_{B_i}h(x,|Du_{j}-Dv_{i,j}|)\,dx \\
& \quad + 3^{q-2}\int_{B_i \cap \{h(x,|Du|)>2\cdot3^{q-2}c_{l}\lambda\}}h(x,|Dv_{i,j}|)\,dx \\
& \le 3^{q-2}\int_{B_i}h(x,|Du-Du_j|)\,dx + 3^{q-2}\int_{B_i}h(x,|Du_j - Dv_{i,j}|)\,dx \\
& \quad + 3^{q-2}c_{l}\lambda|B_{i} \cap \{h(x,|Du|)>2\cdot3^{q-2}c_{l}\lambda\}|
\end{aligned}
\end{equation*}
and therefore
\begin{equation*}
\begin{aligned}
& \int_{B_i \cap \{h(x,|Du|)>2\cdot 3^{q-2}c_{l}\lambda\}}h(x,|Du|)\,dx \\
& \le 2\cdot 3^{q-2}\int_{B_i}h(x,|Du-Du_j|)\,dx + 2\cdot 3^{q-2}\int_{B_i}h(x,|Du_j - Dv_{i,j}|)\,dx \\
& \le 2\cdot 3^{q-2}|4B_i|\left(\mean{4B_i}h(x,|Du-Du_{j}|)\,dx + \mean{2B_i}h(x,|Du_{j}-Dv_{i,j}|)\,dx\right).
\end{aligned}
\end{equation*}
Applying \eqref{approx.Du} and \eqref{final.comp} to each term in the right-hand side, we arrive at
\begin{equation}\label{hDu.Bi}
\int_{B_i \cap \{h(x,|Du|)>2\cdot3^{q-2}c_{l}\lambda\}}h(x,|Du|)\,dx \le 3^{q-2}40^{n}S(R_0,K,M)\lambda|\tilde{B}_i|.
\end{equation}

We now establish a uniform estimate for $|\tilde{B}_i|$. Observe that $\eqref{exit2}_{1}$ implies
\begin{equation*}
|\tilde{B}_i| = \frac{1}{\lambda}\int_{\tilde{B}_i}[h(x,|Du|) + M\mathbf{M}_{1}(\mu)]\,dx.
\end{equation*}
In particular, either
\begin{equation}\label{hm}
\frac{1}{\lambda}\int_{\tilde{B}_i}h(x,|Du|)\,dx \ge \frac{1}{2}|\tilde{B}_i| \qquad \text{or} \qquad \frac{1}{\lambda}\int_{\tilde{B}_i}M\mathbf{M}_{1}(\mu)\,dx \ge \frac{1}{2}|\tilde{B}_i|
\end{equation}
must hold. If $\eqref{hm}_1$ holds, then we have
\begin{equation*} 
|\tilde{B}_i| \le \frac{2}{\lambda}\int_{\tilde{B}_i}h(x,|Du|)\,dx 
\le \frac{2}{\lambda}\int_{\tilde{B}_i \cap \{h(x,|Du|)>\lambda/4\}}h(x,|Du|)\,dx + \frac{1}{2}|\tilde{B}_i|
\end{equation*}
and so
\[ |\tilde{B}_i| \le \frac{4}{\lambda}\int_{\tilde{B}_i \cap \{h(x,|Du|)>\lambda/4\}}h(x,|Du|)\,dx. \]
Similarly, if $\eqref{hm}_2$ holds, then we have
\[ |\tilde{B}_i| \le \frac{4}{\lambda}\int_{\tilde{B}_i \cap \{\mathbf{M}_1(\mu)>\lambda/(4M)\}}M\mathbf{M}_1(\mu)\,dx. \]
The above two displays imply
\begin{equation*}
|\tilde{B}_i| \le \frac{4}{\lambda}\int_{\tilde{B}_i \cap \{h(x,|Du|) > \lambda/4\}}h(x,|Du|)\,dx 
+ \frac{4}{\lambda}\int_{\tilde{B}_i \cap \{\mathbf{M}_{1}(\mu) > \lambda/(4M)\}}M\mathbf{M}_{1}(\mu)\,dx.
\end{equation*}
Connecting this inequality to \eqref{hDu.Bi}, we arrive at
\begin{equation}\label{upper.1st}
\begin{aligned}
& \int_{B_i \cap \{h(x,|Du|)>2\cdot3^{q-2}c_{l}\lambda\}}h(x,|Du|)\,dx \\
& \le 3^{q-2}80^{n}S(R_0,K,M)\int_{\tilde{B}_i \cap \{h(x,|Du|)>\lambda/4\}}h(x,|Du|)\,dx \\
& \quad + 3^{q-2}80^{n}S(R_{0},K,M)\int_{\tilde{B}_i \cap \{\mathbf{M}_{1}(\mu)>\lambda/(4M)\}}M\mathbf{M}_{1}(\mu)\,dx.
\end{aligned}
\end{equation}
Recalling \eqref{covering} and the fact that the balls $\{\tilde{B}_i\}$ are disjoint, after a change of variables, we sum up \eqref{upper.1st} over the covering $\{B_i\}$, thereby obtaining
\begin{equation}\label{upper.2nd}
\begin{aligned}
\int_{E(r_1,\lambda)}h(x,|Du|)\,dx & \le 3^{q-2}80^{n}S(R_0,K,M)\int_{E(r_2,\lambda/(8\cdot3^{q-2}c_{l}))}h(x,|Du|)\,dx \\
& \quad + 3^{q-2}80^{n}S(R_0,K,M)\int_{\mathcal{E}(r_{2},\lambda/(8\cdot3^{q-2}c_{l}M))}M\mathbf{M}_{1}(\mu)\,dx
\end{aligned}
\end{equation}
whenever
\[ \lambda \ge 2\cdot3^{q-2}c_{l}\lambda_{0} = 2\cdot3^{q-2}c_{l}\frac{20^{n}r_{2}^{n}}{(r_{2}-r_{1})^n}\mean{B_{r_2}}[h(x,|Du|) + M\mathbf{M}_{1}(\mu)]\,dx. \]
Here we have used the notation
\[ \mathcal{E}(s,\lambda) \coloneqq \{x \in B_{s}: \mathbf{M}_{1}(\mu)(x) > \lambda \} , \qquad R \le s \le 2R, \;\; \lambda>0 \]
and also used the fact that $\lambda \ge \lambda_0$ for $\lambda_0$ defined in \eqref{def.lambda0}.

\textit{Step 11: Integration and conclusion. } 
We now conclude the proof by integrating on level sets. We consider, for $t \ge 0$, the truncated function
\[ [h(x,|Du|)]_t \coloneqq \min\{h(x,|Du|), t\}. \]
Then, for $t \ge 4\cdot3^{q-2}c_{l}\lambda_0$, we multiply \eqref{upper.2nd} by $\lambda^{\gamma-2}$ and integrate the resulting estimate with respect to $\lambda$ in order to have
\begin{equation}\label{int.start}
\begin{aligned}
& \int_{2\cdot 3^{q-2}c_{l}\lambda_0}^{t}\lambda^{\gamma-2}\int_{E(r_{1},\lambda)}h(x,|Du|)\,dx\,d\lambda \\
& \le 3^{q-2}80^{n}S(R_0,K,M)\int_{2\cdot3^{q-2}c_{l}\lambda_0}^{t}\lambda^{\gamma-2}\int_{E(r_2,\lambda/(8\cdot3^{q-2}c_{l}))}h(x,|Du|)\,dx\,d\lambda \\
& \quad + 3^{q-2}80^{n}S(R_0,K,M)\int_{2\cdot3^{q-2}c_{l}\lambda_0}^{t}\lambda^{\gamma-2}\int_{\mathcal{E}(r_2,\lambda/(8\cdot3^{q-2}c_{l}M))}M\mathbf{M}_1(\mu)\,dx\,d\lambda.
\end{aligned}
\end{equation}
Observe that Fubini's theorem implies
\begin{equation}\label{int.2nd}
\begin{aligned}
& \int_{2\cdot 3^{q-2}c_{l}\lambda_0}^{t}\lambda^{\gamma-2}\int_{E(r_{1},\lambda)}h(x,|Du|)\,dx\,d\lambda \\
& = \frac{1}{\gamma-1}\int_{B_{r_1}}[h(x,|Du|)]^{\gamma-1}_{t}h(x,|Du|)\,dx - \int_{0}^{2\cdot3^{q-2}c_{l}\lambda_0}\lambda^{\gamma-2}\int_{E(r_1,\lambda)}h(x,|Du|)\,dx\,d\lambda
\end{aligned}
\end{equation}
and, in a similar way, 
\begin{equation}\label{int.22}
\begin{aligned}
& \int_{2\cdot3^{q-2}c_{l}\lambda_0}^{t}\lambda^{\gamma-2}\int_{E(r_2,\lambda/(8\cdot3^{q-2}c_l))}h(x,|Du|)\,dx\,d\lambda \\
& \le \frac{(8\cdot3^{q-2}c_l)^{\gamma-1}}{\gamma-1}\int_{B_{r_2}}[h(x,|Du|)]^{\gamma-1}_{t/(8\cdot3^{q-2}c_l)}h(x,|Du|)\,dx \\
& \le \frac{(8\cdot3^{q-2}c_l)^{\gamma-1}}{\gamma-1}\int_{B_{r_2}}[h(x,|Du|)]^{\gamma-1}_{t}h(x,|Du|)\,dx.
\end{aligned}
\end{equation}
Note that, in the above estimate, we have changed variables and used the inequality 
\[ [h(x,|Du|)]^{\gamma-1}_{t/(8\cdot3^{q-2}c_l)} \le [h(x,|Du|)]^{\gamma-1}_{t}, \]
which is valid since $8\cdot3^{q-2}c_l >1$.

As for the integral of $\mathbf{M}_1(\mu)$, we again use Fubini's theorem to have
\begin{equation}\label{int.23}
\begin{aligned}
& \int_{2\cdot3^{q-2}c_{l}\lambda_0}^{t}\lambda^{\gamma-2}\int_{\mathcal{E}(r_2,\lambda/(8\cdot3^{q-2}c_l M))}\mathbf{M}_1(\mu)\,dx\,d\lambda \\
& \le \int_{0}^{\infty}\lambda^{\gamma-2}\int_{\mathcal{E}(r_2,\lambda/(8\cdot3^{q-2}c_l M))}\mathbf{M}_1(\mu)\,dx\,d\lambda \le \frac{(8\cdot3^{q-2}c_l M)^{\gamma-1}}{\gamma-1}\int_{B_{r_2}}[\mathbf{M}_1(\mu)]^{\gamma}\,dx.
\end{aligned}
\end{equation}
Plugging \eqref{int.2nd}, \eqref{int.22} and \eqref{int.23} into \eqref{int.start}, and then making elementary manipulations, we arrive at
\begin{equation}\label{int.3rd}
\begin{aligned}
& \mean{B_{r_1}}[h(x,|Du|)]^{\gamma-1}_{t}h(x,|Du|)\,dx \\
& \le c_{f}^{\gamma}c_{l}^{\gamma-1}S(R_0,K,M)\mean{B_{r_2}}[h(x,|Du|)]^{\gamma-1}_{t}h(x,|Du|)\,dx \\
& \quad + c_{f}^{\gamma}c_{l}^{\gamma-1}M^{\gamma-1}S(R_0,K,M)\mean{B_{r_2}}[\mathbf{M}_1(\mu)]^{\gamma}\,dx + c_{f}^{\gamma}c_{l}^{\gamma-1}\lambda_{0}^{\gamma},
\end{aligned}
\end{equation}
where $c_{l}\equiv c_{l}(\data) \ge 1$ has been given in \eqref{lip.lambda} and $c_{f} \equiv c_{f}(n,q) \ge 1$. Note that we have used the inequality $|B_{r_2}|/|B_{r_1}| \le 2^n$ and that, recalling the definition of $\lambda_0$  in \eqref{def.lambda0}, we have estimated the second term in the right-hand side of \eqref{int.2nd} as
\begin{equation*}
\begin{aligned}
\int_{0}^{2\cdot3^{q-2}c_{l}\lambda_0}\lambda^{\gamma-2}\int_{E(r_1,\lambda)}h(x,|Du|)\,dx\,d\lambda & \le \int_{0}^{2\cdot3^{q-2}c_{l}\lambda_0}\lambda^{\gamma-2}\,d\lambda \int_{B_{r_2}}h(x,|Du|)\,dx \\
& \le \frac{(2\cdot3^{q-2}c_l)^{\gamma-1}\lambda_{0}^{\gamma-1}}{\gamma-1}\int_{B_{r_2}}h(x,|Du|)\,dx \\
& \le \frac{(2\cdot3^{q-2})^{\gamma-1}c_{l}^{\gamma-1}}{\gamma-1}\lambda_{0}^{\gamma}|B_{r_2}|.
\end{aligned}
\end{equation*}

We now deal with the quantity $S(R_0,K,M)$ defined in \eqref{def.S}; note that all the above estimates are valid for any choices of $R_0 \in (0,1)$ satisfying \eqref{R0.ini} and \eqref{R0.ini2}, $K \ge 4$ and $M \ge 1$, with all the constants, except for the ones in the definition  of $S(R_0,K,M)$, remaining uniformly bounded and ultimately depending only on $\data$. 
Here we fix $R_0$, $K$ and $M$ in order to have
\begin{equation}\label{fix.para}
c_{f}^{\gamma}c_{l}^{\gamma-1}S(R_0,K,M) \le \frac{1}{2}.
\end{equation}
More precisely, we start fixing $K \equiv K(\data,\gamma) \ge 4$ as
\begin{equation}\label{fix.K}
K \coloneqq \left(2^{4q}c_{f}^{\gamma}c_{l}^{\gamma}\tilde{c}\right)^{q/(p-1)},
\end{equation}
where $\tilde{c}\equiv \tilde{c}(n,p,q,\nu,L)$ appears in \eqref{comp.comb}. This in turn determines the constant $c_K$ in \eqref{comp.comb} as a function of $\data$ and $\gamma$ only. We then choose $M \equiv M(\data,\gamma) \ge 1$ as
\begin{equation}\label{fix.M}
M \coloneqq 16c_{f}^{\gamma}c_{l}^{\gamma}c_{0}c_{*}.
\end{equation}
We finally choose $R_0 \equiv R_0 (\data,|\mu|(\Omega),\gamma) \in (0,1)$ small enough to have \eqref{R0.ini} and 
\begin{equation}\label{fix.R0}
R_0 \le \left(\frac{1}{2^{4q}c_{f}^{\gamma}c_{l}^{\gamma}c_{K}}\right)^{1/\sigma},
\end{equation}
where $\sigma \equiv \sigma (n,p,q,\alpha)$ is given in \eqref{def.sigma}; in particular, we also have \eqref{R0.ini2}. 
These choices of parameters in \eqref{fix.K}, \eqref{fix.M} and \eqref{fix.R0} yield \eqref{fix.para}, which with \eqref{int.3rd} in turn leads to
\begin{equation*}
\begin{aligned}
& \mean{B_{r_1}}[h(x,|Du|)]^{\gamma-1}_{t}h(x,|Du|)\,dx \\
& \le \frac{1}{2}\mean{B_{r_2}}[h(x,|Du|)]^{\gamma-1}_{t}h(x,|Du|)\,dx + c_{\gamma}\mean{B_{2R}}[\mathbf{M}_1(\mu)]^{\gamma}\,dx \\
& \quad + \frac{c^{\gamma}R^{n\gamma}}{(r_2 - r_1)^{n\gamma}}\left(\mean{B_{2R}}[h(x,|Du|)+M\mathbf{M}_1(\mu)]\,dx \right)^{\gamma},
\end{aligned}
\end{equation*}
where $c\equiv c(\data)$, $c_{\gamma} \equiv c_{\gamma}(\data,\gamma)$ and $M \equiv M(\data,\gamma)$; we have also used the definition of $\lambda_0$ in \eqref{def.lambda0}. Moreover, we note that the constants in the above estimate are independent of the parameter $t \ge 4\cdot 3^{q-2}c_{l}\lambda_0$. 

We are now in a position to apply Lemma \ref{tech.lemma} with the choices $\ell = n\gamma$ and
\[ Z(s) \coloneqq \mean{B_s}[h(x,|Du|)]^{\gamma-1}_{t}h(x,|Du|)\,dx, \qquad R \le s \le 2R, \]
which is obviously a bounded function. Then, making a few elementary manipulations together with Jensen's inequality and recalling that $M \ge 1$ has been determined in \eqref{fix.M} as a function of $\data$ and $\gamma$, we arrive at
\[ \mean{B_R}[h(x,|Du|)]^{\gamma-1}_{t}h(x,|Du|)\,dx \le c^{\gamma}\left(\mean{B_{2R}}h(x,|Du|)\,dx\right)^{\gamma} + c_{\gamma}\mean{B_{2R}}[\mathbf{M}_1(\mu)]^{\gamma}\,dx. \]
Letting $t\to\infty$ in the above display, we conclude with
\[ \mean{B_R}[h(x,|Du|)]^{\gamma}\,dx \le c^{\gamma}\left(\mean{B_{2R}}h(x,|Du|)\,dx\right)^{\gamma} + c_{\gamma}\mean{B_{2R}}[\mathbf{M}_1(\mu)]^{\gamma}\,dx \]
whenever $R \le R_0$ with $R_0 \equiv R_0(\data,|\mu|(\Omega),\gamma) \in (0,1)$ satisfying \eqref{R0.ini} and \eqref{fix.R0}, and where $c\equiv c(\data)$ and $c_{\gamma}\equiv c_{\gamma}(\data,\gamma)$. This shows \eqref{main.est}, which along with a standard covering argument leads to \eqref{cz.implication}. The proof is finally complete. \qed


\subsection*{Conflict of interest} The authors declare that they have no conflict of interest.

\subsection*{Data availability} Data sharing not applicable to this article as no datasets were generated or analyzed during the current study.

\end{document}